\documentclass[reqno]{amsart}

\newtheorem{theoremletter}{Theorem}

\usepackage{mathrsfs}
\usepackage{hyperref}
\usepackage{pdfsync}
\usepackage{todonotes}
\usepackage{amssymb,amsfonts,amsmath,amsthm}
\usepackage{graphicx}
\usepackage{color}
\usepackage{verbatim}
\usepackage{wasysym}

\newtheorem{thm}{Theorem}[section]
\newtheorem{pro}[thm]{Proposition}
\newtheorem{lem}[thm]{Lemma}
\newtheorem{cor}[thm]{Corollary}

\theoremstyle{definition}
\newtheorem{definition}[thm]{Definition}
\newtheorem{exa}[thm]{Example}

\newtheorem*{ackn}{Acknowledgements}

\theoremstyle{remark}
\newtheorem{rmk}[thm]{Remark}
\newtheorem*{claim}{Claim}
\newtheorem{prb}[thm]{Problem}
\newtheorem*{prob}{Problem}

\numberwithin{equation}{section}

\def\EC{effectively closed for rational intersections}

\def\eps{\varepsilon}
\def\es{\varnothing}

\def\ol#1{\overline{#1}}

\def\sh#1{\mathrm{SH}(#1)}
\def\gen#1{\langle{#1}\rangle}
\def\pre#1#2{\langle{#1}\,|\,{#2}\rangle}

\DeclareMathOperator\Mon{Mon} \DeclareMathOperator\Inv{Inv} \DeclareMathOperator\Gp{Gp} 
\DeclareMathOperator\pref{pref} \DeclareMathOperator\suff{suff} \DeclareMathOperator\red{red} 
\DeclareMathOperator\Rat{Rat}

\begin{document}


\title[The prefix membership problem for one-relator groups]%
{New results on the prefix membership problem for one-relator groups} 


\author{IGOR DOLINKA}

\address{Department of Mathematics and Informatics, University of Novi Sad, Trg Dositeja Obradovi\'ca 4,
21101 Novi Sad, Serbia}

\email{dockie@dmi.uns.ac.rs}

\author{ROBERT D. GRAY}

\address{School of Mathematics, University of East Anglia, Norwich NR4 7TJ, England, UK}

\email{Robert.D.Gray@uea.ac.uk}

\thanks{The research of the first named author was supported by the Ministry of Education, Science, and Technological
Development of the Republic of Serbia through the grant No.174019. The research of the second named author was supported 
by the EPSRC grant EP/N033353/1 ``Special inverse monoids: subgroups, structure, geometry, rewriting systems and the word problem''}








\begin{abstract}
In this paper we prove several results regarding decidability of the membership problem for 
certain submonoids in amalgamated free products and HNN extensions of groups. These general results
are then applied to solve the prefix membership problem for a number of classes of one-relator groups
which are low in the Magnus--Moldavanski\u{\i} hierarchy. Since the prefix membership problem 
for one-relator groups is intimately related to the word problem for one-relator special 
inverse monoids in the $E$-unitary case (as discovered in 2001 by Ivanov, Margolis and Meakin),
these results yield solutions of the word problem for several new classes of one-relator special 
inverse monoids. 
In establishing these results, we introduce a new theory of 
conservative factorisations of words 
which provides a link between the prefix membership problem of a one-relator group and the
group of units of the corresponding one-relator special inverse monoid. Finally, we exhibit the
first example of a one-relator group, defined by a reduced relator word, that has an undecidable prefix
membership problem.
\end{abstract}


\maketitle

\section{Introduction}
\label{sec:intro} 

From the early days of combinatorial group theory, algorithmic problems have occupied a central position in the course of its development, going back to the pioneering work of Dehn \cite{Dehn}.
One of the most celebrated classical results in this area is the positive solution of the word problem for one-relator groups by Dehn's student Magnus \cite{Ma2}. It is based on a previous important result by 
Magnus \cite{Ma1}, the \emph{Freiheitssatz}, stating that if $w$ is a cyclically reduced word then any subgroup 
of the one-relator group $\Gp\pre{X}{w=1}$ generated by a subset of $X$ omitting at least one letter that appears in $w$ must be free. Magnus's approach is applicable to an array of other algorithmic problems for 
one-relator groups and entails what is today known as the \emph{Magnus method}. The modern exposition of this 
method stems from the paper of McCool and Schupp \cite{McCSch} (see also the monograph \cite{LSch}), and is 
based on an observation due to Moldavanski\u{\i} \cite{Mol} that if $w$ is a reduced word and has exponent sum 
zero for some letter from $X$, then $\Gp\pre{X}{w=1}$ is an HNN extension of a one-relator group with a defining 
relator shorter than $w$.

Given the vigorous development of combinatorial algebra over a number of decades, it is quite striking 
that the following problem still remains open.

\begin{prob}
Is the word problem decidable for all one-relation monoids $\Mon\pre{X}{u=v}$ (where $u,v$ are words over $X$)?
\end{prob}

\noindent 
This problem has received significant attention, with a number of special cases being solved. 
A  strong early impetus was given by Adjan \cite{Adj} who proved that $\Mon\pre{X}{u=v}$ has
decidable word problem if either one of the words $u,v$ are empty (this is the case of the so-called \emph{special
monoids}, with presentations of the form $\Mon\pre{X}{w=1}$), or both $u,v$ are non-empty and have different
initial and different terminal letters. For both of these cases, Adjan exhibits a reduction of the monoid
word problem to the word problem of an associated one-relator group, and then makes an appeal to Magnus's result.
Later on, Adyan and Oganessyan \cite{AO} showed that the word problem for $\Mon\pre{X}{u=v}$
can be reduced to the case
of monoid presentations of the form $\Mon\pre{X}{asb=atc}$, where $a,b,c\in X$, $b\neq c$, and $s,t$ are 
arbitrary words over $X$.

An entirely new approach to the problem was provided by the work of Ivanov, Margolis and Meakin \cite{IMM}
(which is also the central reference for the present paper). 
There, a crucial observation is made that the monoid $\Mon\pre{X}{asb=atc}$ (arising from the reduction found in \cite{AO}) embeds into the \emph{inverse monoid} defined
by the inverse monoid presentation $\Inv\pre{X}{asbc^{-1}t^{-1}a^{-1}=1}$; consequently, the decidability of the word
problem for special inverse monoid presentations $\Inv\pre{X}{w=1}$ (where $w$ is a word over the 
alphabet $\ol{X}=X\cup X^{-1}$) would immediately imply the positive solution of the word problem for one-relator 
monoids. This strongly motivates the study of special inverse monoids and their word problems, which is also
interesting in its own right, given the prevalence of inverse semigroups and their combinatorial and geometric 
aspects in various areas of mathematics (see \cite{Law}). However, a recent surprising result of Gray 
\cite{Gr-Inv} shows that the word problem for one-relator special inverse monoids in complete generality is
\emph{undecidable}. 

Given that the word problem for $\Inv\pre{X}{w=1}$ is undecidable in general, the key problem that remains is to determine for which words $w \in (X \cup X^{-1})^*$ it is decidable? 
In particular, is it decidable if $w$ is (i) a reduced word or (ii) a cyclically reduced word? 
A positive answer to the first of these questions would still, as a consequence of the 
results from \cite{IMM} described above, imply a positive answer to decidability of the word problem for arbitrary one-relator monoids $\Mon\pre{X}{u=v}$.  
This motivates investigating the word problem in the cases that $w$ is reduced or cyclically reduced. 

In the cyclically reduced case, the word problem for the one-relator inverse monoid $\Inv\pre{X}{w=1}$ is closely related to an algorithmic problem in the corresponding one-relator group  $\Gp\pre{X}{w=1}$ called the prefix membership problem. 
For a one-relator group $G=\Gp\pre{X}{w=1}$, let $P_w$ denote the submonoid of $G$ generated by the elements of $G$ represented by all prefixes of $w$. This is the \emph{prefix monoid} of $G$. Another crucial result from \cite{IMM}  (Theorem 3.1), shows that if the inverse monoid $\Inv\pre{X}{w=1}$ has the so-called $E$-unitary property (which  is e.g.\ the case when $w$ is cyclically reduced) then the word problem of $\Inv\pre{X}{w=1}$ is decidable whenever the membership problem for $P_w$ in $G$ is decidable. 
This is significant because it translates the word problem for a one-relator special inverse monoids 
into the realm of one-relator groups and associated decision problems. 

The connections between decision problems for monoids, inverse monoids and groups just described highlight
the importance of other, more general, algorithmic problems. For example, it is still
unknown whether the \emph{subgroup membership problem}---also called the \emph{generalised word problem}---is
decidable for one-relator groups. However, there exist one-relator groups in which the \emph{submonoid membership
problem} (and thus the more general \emph{rational subset membership problem} \cite{Loh}) is undecidable \cite{Gr-Inv}. 
The one-relator group with undecidable submonoid membership problem given in \cite{Gr-Inv} is an HNN extension of $\mathbb{Z} \times \mathbb{Z}$ 
with respect to an isomorphism mapping one of the natural copies of $\mathbb{Z}$ to the other.
So in general the submonoid and rational subset membership problems are not well-behaved under the HNN extension construction, and similarly for free products with amalgamation.  
On the other hand, under the assumption of finiteness of edge groups, the decidability of the rational subset 
membership problem is preserved under the graph of groups construction \cite{KSS}, which includes amalgamated 
free products and HNN extensions. We direct the reader e.g.\ to \cite{Gru,KWM,Loh,LSt} for a sampler of results 
in this broader area in which the present topic is couched.

Motivated by the above discussion, both the word problem for one-relator inverse monoids and the prefix membership problem for one-relator groups, with cyclically reduced defining relator, have already received a great deal of attention in the literature; see e.g. \cite{Hermiller:2010bs, Inam, IMM, Juh, MMSu, Mea} and   
\cite[Question~13.10]{Bes-list}.
In this paper we will make several new contributions towards resolving these open problems.  
The new approaches to these problems that we present in this article naturally divide into two themes.  

Firstly, as mentioned above, the standard modern approach to proving results about one-relator 
a one-relator group $\Gp\pre{X}{w=1}$  is by induction on the length of $w$ using the McCool--Schupp \cite{McCSch, LSch} Moldavanski\u{\i} \cite{Mol} approach via HNN extensions.
The inductive step of this approach is based on the fact that the one-relator group embeds in a certain HNN extension of a one-relator group with a shorter defining relator.  
We shall refer to the steps in this induction as levels in the Magnus--Moldavanski\u{\i} hierarchy.
Given its utility in proving other results for one-relator groups, it is very natural to also attempt to use this approach to investigate the prefix membership problem for $\Gp\pre{X}{w=1}$.  If the group happens to be free then by a theorem of  Benois \cite{Be} the prefix membership problem is decidable (in fact, the more general rational subset membership problem is decidable for free groups.).   So the next natural step in this approach is to investigate the prefix membership problem for one-relator groups that are one-step away from being free in this hierarchy. The general results we prove for HNN extensions in this paper are motivated by this idea, and we will apply them in this paper to prove decidability of the prefix membership problem for several classes of one-relator groups which are low in the Magnus--Moldavanski\u{\i} hierarchy.

The second new viewpoint revealed by the results we prove in this paper is that  
the word problem in $\Inv\pre{A}{w=1}$ can be often be shown to be decidable by analysing decompositions of the word 
$w \equiv w_1 w_2 \dots w_k$, where all the $w_i$ represent invertible elements of the monoid.
We call this a \emph{unital decomposition} of the word $w$.  
We shall identify several combinatorial conditions on unital decompositions which suffice to imply decidability of the word problem for the monoid. This gives a new approach to the word problem for one-relator inverse monoids which goes via the group of units, in this sense.  Something that makes this approach widely applicable is that the above decomposition of $w$ does not need to be minimal in order for our results to apply.  
That is, provided the words $w_i$ satisfying the needed combinatorial properties, it is not important whether or not there is a finer decomposition of $w$ as a product of units.   
This means that there are situations where we can show the word problem is decidable without necessarily having an algorithm to compute the minimal invertible pieces of the defining relator word.
Similarly it means that the word problem can sometimes be shown to be decidable without having to determine the structure of the group of units of the monoid.  
To make use of information about unital decompositions in the inverse monoid presentation to solve the prefix membership problem in the corresponding group $\Gp\pre{A}{w=1}$ we develop a theory of, so-called, \emph{conservative factorisations} of relator words in one-relator groups.
This is another key new idea that we introduce in this paper, since it allows us to transform algebraic information about units in the inverse monoid into corresponding algebraic information about submonoids of the maximal group image generated by prefixes of pieces of the relator.     
This allows us to state our results entirely in terms of one-relator groups and conservative factorisations, 
and then apply them to solve the word problem for various families of one-relator inverse monoids.   

These new approaches give rise to results which, when expressed in their most general form, prove 
decidability of the membership problem in 
certain submonoids of amalgamated free products of groups and HNN extensions of groups. 
In this paper we prove four new general results of this kind.  Specifically, we   prove two general theorems for amalgamated free products in Section~\ref{sec:amalg}, Theorems~\ref{thm:amal} and  \ref{thm:amal51}, and then in Section~\ref{sec:hnn} we prove two general theorems for HNN extensions of groups,  namely Theorem~\ref{thm:hnn41} and Theorem~\ref{thm:hnn}. Then, in Sections~\ref{sec:appl1} and \ref{sec:appl2}, respectively, we show how, via ideas summarised in the description of the two main themes above,   
we can apply these general results to solve the prefix membership problem for certain  one-relator groups and, consequently, to solve the word problem for some classes of special one-relator inverse monoids. 
As applications we recover new proofs of numerous results from the literature 
\cite{BMM,IMM,Juh,MM93,MMSu, Mea} (bringing them under a common framework), and at the same time we 
prove decidability of the prefix membership problem for 
many classes of one-relator groups (and special one-relator inverse monoids) not covered by previous results. In particular, our work was inspired by the attempts to solve the word problem for the so-called \emph{O'Hare monoid} (see \cite{MMSt,Mea} and Example~\ref{exa:ohare} below), which is eventually dealt with  in this paper, in Proposition~\ref{pro:OHare}. Other main applications of our general results include  Theorems~\ref{thm:marker}, \ref{thm:disj}, \ref{thm:pinch1}, \ref{thm:pos-neg}, \ref{thm:adjan} and  \ref{thm:pinch2}.

In the last section of the paper we present a result of a different flavour which 
says something about the limits of what we should hope to be able to prove about the prefix membership problem in one-relator groups.  
Specifically, by modifying the construction from \cite{Gr-Inv}, we will show in Theorem~\ref{thm:undec} that there is a finite alphabet $X$ and a reduced word $w \in (X \cup X^{-1})^*$ such that $\Gp\pre{X}{w=1}$ has undecidable prefix membership problem. 
Hence if \cite[Question~13.10]{Bes-list} has a positive answer then the cyclically reduced hypothesis will need to be used. 

The paper is organised as follows. 
In the next preliminary section we gather the notation and basic notions, aiming to make the paper reasonably 
self-contained. This is followed by Section~\ref{sec:factor} where we discuss the relationship between two 
types of factorisations of a word $w$ appearing as a relator in $M=\Inv\pre{X}{w=1}$, namely, \emph{unital} ones, 
decomposing $w$ into pieces representing invertible elements (units) of the inverse monoid $M$, and 
\emph{conservative} factorisations preserving, in a sense, the prefix monoid $P_w$. When $M$ is $E$-unitary,
these two types of factorisation coincide (see Theorem \ref{thm:UnitalConservative}), and that, taken together 
with the Benois factorisation algorithm devised by Gray and Ru\v skuc \cite{GR} (producing such factorisations 
in a manner finer than the Adjan-Zhang overlap algorithm \cite{Adj,Zh1}), is an important pre-requisite for some 
of our decidability results. The main body of our results is then presented in 
Sections~\ref{sec:amalg}--\ref{sec:appl2}. We finish the paper by few concluding remarks in 
Section~\ref{sec:conclude}.


\section{Preliminaries}
\label{sec:prelim}

We give some background definitions and results from combinatorial group and monoid theory that will be needed 
later. For more background we refer the reader to \cite{LSch} for groups, \cite{How,Law,Pet} for monoids and 
inverse semigroups, and \cite{HU,Pin} for automata and formal languages. In particular we refer the reader to 
\cite{LSch} for basic notions from the algorithmic theory of finitely generated groups.  

\subsection{Words and free objects}

Let $X$ be a finite alphabet. By $X^*$ we denote the \emph{free monoid} on $X$, consisting of all words over $X$
including the empty word $1$. However, whenever we are concerned with groups and inverse monoids it is
more useful to consider a `doubled' alphabet $\ol{X}=X\cup X^{-1}$, where $X^{-1}=\{x^{-1}:\ x\in X\}$ is a 
disjoint copy of $X$, with an obvious bijective correspondence between $X$ and $X^{-1}$. 
Now the free monoid $\ol{X}^*$ has a natural
involutory operation so that for a word $w=x_1^{\eps_1}\dots x_k^{\eps_k}$, $x_1,\dots, x_k\in X$,
$\eps_1,\dots,\eps_k\in\{-1,1\}$, we have $w^{-1}=x_k^{-\eps_k}\dots x_1^{-\eps_1}$. 

When $w\in \ol{X}^*$, we use the notation $w(x_1,\dots,x_n)$ to stress that the letters occurring in $w$ are
among $x_1,\dots,x_n,x_1^{-1},\dots,x_n^{-1}$. In other words, an occurrence of a letter $x_i$ in $w$ can happen 
either as $x_i$, or as $x_i^{-1}$. Given $w(x_1,\dots,x_n)$ and a sequence of (not necessarily distinct) words
$p_1,\dots,p_n\in \ol{X}^*$, we write $w(p_1,\dots,p_n)$ to denote the word obtained from $w=w(x_1,\dots,x_n)$
by replacing each letter $x_i$ by $p_i$ and each letter $x_i^{-1}$ by $p_i^{-1}$.

Given $w\in \ol{X}^*$ we denote by $\red(w)$ the \emph{reduced form} of $w$, which is obtained
from $w$ by the confluent rewriting process of successively removing subwords of the form $xx^{-1}$ and
$x^{-1}x$, where $x\in X$. This notation is extended to sets, too, so that for $A\subseteq \ol{X}^*$, $\red(A)$
stands for the set of words obtained by reducing each word from $A$. As is well known, one can identify the
elements of the free group $FG(X)$ on $X$ with the set of all reduced words from $\ol{X}^*$, so that the result
of the multiplication of two such words $u,v$ is $\red(uv)$, and the inverse of $u$ is simply $u^{-1}$.

A monoid $M$ is called \emph{inverse} \cite{Law,Pet} if for every $a\in M$, 
there is a unique element $a^{-1} \in M$, called the inverse of $a$, such that  
$aa^{-1}a=a$ and $a^{-1}aa^{-1}=a^{-1}$. 
Inverse monoids form a variety in
the sense of universal algebra, so free inverse monoids $FIM(X)$ exist. A straightforward, albeit implicit 
description of $FIM(X)$ is given as a quotient of $\ol{X}^*$ by the so-called \emph{Vagner congruence}: this
is the congruence of $\ol{X}^*$ generated by all pairs of the form $(u,uu^{-1}u)$ and $(uu^{-1}vv^{-1}, 
vv^{-1}uu^{-1})$, where $u,v\in \ol{X}^*$. Concrete descriptions of $FIM(X)$ (and so the solutions of 
its word problem) go back to Scheiblich \cite{Sch} and Munn \cite{Munn}: the element of $FIM(X)$ represented
by a word $w\in \ol{X}^*$ can be identified with a birooted tree today called the \emph{Munn tree} of $w$. 
This is obtained as a connected subtree of the Cayley tree of the free group $FG(X)$ which arises by travelling
along the path labelled by $w$.
Hence, $u,v\in \ol{X}^*$ represent the same element of $FIM(X)$ if and only if they give rise to the same
Munn tree. Clearly, there is a natural surjective homomorphism $FIM(X)\to FG(X)$.

\subsection{Presentations}

We denote by 
$$
G=\Gp\pre{X}{w_i=1\ (i\in I)}
$$
the group presented by generators $X$ and relators $w_i$, $i\in I$; as usual, this is canonically the quotient of
the free group $FG(X)$ by its smallest normal subgroup $N$ containing all the elements (reduced words) $w_i$, $i\in I$.
Similarly, the monoid defined by a monoid presentation $M=\Mon\pre{X}{u_i=v_i\ (i\in I)}$ is the quotient of the
free monoid $X^*$ by the congruence $\rho$ generated by the pairs $(u_i,v_i)$, $i\in I$.

In an analogous fashion, inverse monoids can be given by inverse monoid presentations
$$
M=\Inv\pre{X}{u_i=v_i\ (i\in I)},
$$
where $M\cong FIM(X)/\rho$ for the inverse monoid congruence $\rho$ of $FIM(X)$ generated by the pairs $(u_i,v_i)$, 
$i\in I$. 
This is equivalent to saying that $M$ as the quotient $\ol{X}^*/\rho'$ 
where $\rho'$ is
the smallest congruence containing the Vagner congruence and all the pairs $(u_i,v_i)$, $i\in I$. When one of
the sides of each defining relation is the empty word, say $v_i$ is the empty word for all $i\in I$, we get the notion
of a \emph{special inverse monoid} and special inverse monoid presentations. 
The maximal group homomorphic image of 
$M=\Inv\pre{X}{u_i=1\ (i\in I)}$ is the group defined by the presentation $\Gp\pre{X}{u_i=1\ (i\in I)}$.

For a monoid or inverse monoid $M$, we denote by $U_M$ the group of units  of $M$. 
So $U_M$ is the set of all invertible elements of the monoid $M$.   
If $G$ is a group and $A\subseteq G$, we denote by $\Gp\gen{A}$ the \emph{subgroup} generated by $A$, while
$\Mon\gen{A}$ is the \emph{submonoid} of $G$ generated by $A$.

Throughout the paper, if $M$ is an inverse monoid generated by a set $X$, given any two words $u, v \in \ol{X}^*$ we say 
$u = v$ in $M$ to mean that that the two words represent the same element of the inverse monoid, and write $u \equiv v$ to mean that $u$ and $v$ are identical as words in $\ol{X}^*$.   
The same comments apply in particular when we are working with a group $G$ generated by a set $X$. 
Also, in this situation, given any subset $A$ of $\ol{X}^*$ by the submonoid of $G$ generated by $A$, we mean the submonoid generated by the set of all elements of $G$ represented be words in $A$ (that is, the image of $A$ in $G$). We write this as $\Mon\gen{A} \leq G$. Similarly we talk about the subgroup $\Gp\gen{A}$ of $G$ generated by the set of words $A$.

\subsection{$E$-unitary inverse monoids}

Let $M$ be an inverse monoid and $A\subseteq M$. The subset $A$ is said to be \emph{left unitary} 
if $a\in A$, $s\in M$ and $as\in A$ imply $s\in A$. 
The notion of \emph{right unitary} subset is defined dually. 
A subset is \emph{unitary}
if it is both left and right unitary. As is shown, for example, in \cite[Proposition 2.4.3]{Law}, when $A$
is $E=E(M)$, the set of idempotents of $M$, the properties of being left, right and two-sided $E$-unitary
coincide, thus defining the class of \emph{$E$-unitary inverse monoids}.

Each inverse monoid $M$ has the \emph{minimum group congruence} $\sigma$, the smallest congruence of $M$ such that
$M/\sigma$ is a group (see \cite[Theorem 2.4.1]{Law}). On the other hand, on any inverse monoid $M$ one can
define the \emph{compatibility relation} $\sim$ by
$$
a\sim b\text{ if and only if }ab^{-1},a^{-1}b\in E(M).
$$
Whenever $M$ is $E$-unitary, the relation $\sim$ is an equivalence relation, and, furthermore, a congruence of $M$.
In general, $\sigma$ is the congruence generated by the relation $\sim$. In fact, the following characterisation 
holds (see \cite[Theorem 2.4.6]{Law}).

\begin{pro}
An inverse monoid $M$ is $E$-unitary if and only if $\sigma=\,\sim$.
\end{pro}

Turning to the case of special inverse monoids $M=\Inv\pre{X}{u_i=1\ (i\in I)}$, we have that $M/\sigma=G=
\Gp\pre{X}{u_i=1\ (i\in I)}$, and $\sigma$ is simply the kernel relation of the natural homomorphism $M\to G$.
Therefore, we immediately get the following well-known result.

\begin{lem}\label{lem:sim}
Assume that the inverse monoid $M=\Inv\pre{X}{u_i=1\ (i\in I)}$ is $E$-unitary, and let $u,v\in \ol{X}^*$ be
such that $u=v$ holds in $G=\Gp\pre{X}{u_i=1\ (i\in I)}$. Then $u\sim v$ holds in $M$.
\end{lem}

Let us repeat the main result of \cite{IMM}, which confirmed a conjecture formulated
earlier in \cite{MMSt}.

\begin{thm}\emph{(}\cite[Theorem~4.1]{IMM}\emph{)}\label{thm:cr-euni}
If the word $w\in \ol{X}^*$ is cyclically reduced then the inverse monoid $M=\Inv\pre{X}{u=1}$ is $E$-unitary.
\end{thm}

In \cite{IMM} one can find an example of a special inverse monoid with more than one defining relation, and with both defining relators being cyclically reduced words, which is
non-$E$-unitary, so the one-relator condition is essential here.

\subsection{Decision problems in finitely generated groups, finite state automata, the Benois Theorem}

Let $G=\gen{X}$ be a finitely generated group with  
canonical homomorphism $\pi:\ol{X}^*\to G$,
let $A$ be a finite subset of $\ol{X}^*$ 
and let $M=\Mon\gen{A}$ be the submonoid of $G$ generated by $A$. 
The  \emph{membership problem for $M$ in $G$} is the following decision problem:

\medskip

\noindent INPUT: A word $w\in\ol{X}^*$.\\
QUESTION: $w\pi \in M$? 

\medskip

\noindent In other words, is $w$ equal in $G$ to some product of words from $A$?

Given a one-relator group $G=\Gp\pre{X}{w=1}$ we define the associated \emph{prefix monoid} $P_w$ to be the submonoid 
\[
P_w = \Mon\gen{\pref(w)} \leq G, 
\]
where $\pref(w)$ denotes the set of all prefixes of the word $w$. We use $\suff(w)$ to denote the set of all suffixes of $w$. 
We stress that the prefix monoid of $G$ actually \emph{depends on the presentation} of $G$---it can happen that two different 
presentations define the same group, while the corresponding prefix monoids are different, as shown in the following simple example.

\begin{exa}
Both groups $G_1=\Gp\pre{a,b}{aba=1}$ and $G_2=\Gp\pre{a,b}{baa=1}$ are infinite cyclic, that is, free groups of rank $1$
generated by $a$. In this sense, these two presentations define the same group $G=G_1=G_2$. However, the prefix monoid corresponding 
to the first presentation is $M_1=\Mon\gen{a,ab}=\Mon\gen{a,a^{-1}}=G$, while the prefix monoid for the second presentation is
$M_2=\Mon\gen{b,ba}=\Mon\gen{a^{-2},a^{-1}}=\Mon\gen{a^{-1}}=\{1,a^{-1},a^{-2},a^{-3},\dots\}$, clearly a proper submonoid of $M_1$.
\end{exa}

Therefore, we are always going to refer to the prefix monoid of a one-relator group defined by a explicitly stated presentation
$G=\Gp\pre{X}{w=1}$, or make sure that the presentation for $G$ is clear from the context.
Proceeding in this vein, we say that the one-relator group $G$ defined by the presentation $\Gp\pre{X}{w=1}$ has 
decidable \emph{prefix membership problem} if the membership problem for $P_w$ in $G$ is decidable. 
The following crucial connection to the word problem of one-relator special inverse monoids was made in \cite{IMM}.
It is an immediate consequence of 
\cite[Theorem 3.3]{IMM}.

\begin{thm}\emph{(}\cite{IMM}\emph{)}\label{thm:pw-wp}
Let $w\in\ol{X}^*$ be a word such that the inverse monoid 
$$M=\Inv\pre{X}{w=1}$$ 
is $E$-unitary. If the prefix membership problem for $G=\Gp\pre{X}{w=1}$ is decidable, so is the word problem for $M$.
\end{thm}

Let $M=\Inv\pre{X}{w=1}$ and let $R$ be the set of right units of $M$. Then $R$ is a submonoid, but in general not an inverse submonoid, of $M$. 
Clearly, every prefix of $w$ represents an element of $R$. Conversely, by the geometric argument given in the second paragraph of the proof of 
\cite[Proposition 4.2]{IMM}, for every word $u\in\ol{X}^*$ representing an element of $R$ there are prefixes $p_1,\dots,p_k$ of $w$ such that
$u = p_1\dots p_k$ holds in $M$. 
The statement of this proposition in \cite{IMM} actually assumes that $w$ is
cyclically reduced; but it is straightforward to check that the corresponding argument does not make use of that assumption. Hence, in general 
in a special one-relator inverse monoid $M=\Inv\pre{X}{w=1}$ every right unit can be expressed as a product of prefixes of $w$. Therefore, 
the image of $R$ under the natural homomorphism $M\to G=\Gp\pre{X}{w=1}$ is precisely $P_w$.

We use a standard model for \emph{finite state automata} (FSA): this is a quintuple $\mathcal{A}=(Q,\Sigma,E,
I,T)$, where $Q$ is a finite set of states, $\Sigma$ is the alphabet, $I,T\subseteq Q$ are the initial and final states,
respectively, and $E\subseteq Q\times\Sigma\times Q$ are the transitions. 
The automaton $\mathcal{A}$ accepts the word $w\in\Sigma^*$
if there is a path
from an initial state to a final state labelled by $w$; the set of all accepted words $L(\mathcal{A})
\subseteq\Sigma^*$ is the \emph{language} of the FSA. By Kleene's Theorem \cite{HU}, the class of languages of FSA is 
precisely the class of regular languages. 

Given a group $G$, the class of \emph{rational subsets} of $G$ is the smallest set containing all finite subsets of $G$
that is closed with respect to union, product and submonoid generation. 
Note that it is immediate from this definition that every finitely generated submonoid $M$ of $G$ is a rational subset of $G$. 
Combining this notion with Kleene's Theorem,
it is immediate to arrive at the following result.

\begin{pro}
Let $G=\Gp\gen{X}$ be a finitely generated group and let $\pi:\ol{X}^*\to G$ be the corresponding canonical homomorphism.
A subset $R\subseteq G$ is a rational if and only if there is a FSA $\mathcal{A}$ over $\ol{X}$ such that 
$R=L(\mathcal{A})\pi$.
\end{pro}

We note that the empty set is a regular language, and the empty set is a rational subset of $G$ for any group $G$.  

This proposition shows that FSA are convenient vehicles to define rational subsets in finitely generated groups by a finite amount of data.  The \emph{rational subset membership problem} \cite{Loh} for a finitely generated group $G=\Gp\gen{X}$ with the canonical homomorphism $\pi:\ol{X}^*\to G$ is the following decision problem.

\medskip

\noindent INPUT: A FSA $\mathcal{A}$ over $\ol{X}$ and a word $w\in\ol{X}^*$.\\
QUESTION: $w\pi\in L(\mathcal{A})\pi$?

\medskip

A particularly pleasant algorithmic properties are enjoyed by free groups, as a consequence of a key result due to 
Benois \cite{Be} (see also e.g.\ \cite{BS,BMM}).

\begin{thm}\emph{(}\cite{Be}\emph{)}\label{thm:ben}
If $L\subseteq \ol{X}^*$ is a regular language over $\ol{X}$ then $\red(L)$ is also a regular language.
\end{thm}

\begin{cor}\label{cor:ben}
Let $X$ be a finite set. 
Then the free group $FG(X)$ has decidable rational subset membership problem. 
In particular, $FG(X)$ has decidable submonoid membership problem and subgroup membership problem.   
Also, the rational subsets of $FG(X)$ are closed for intersection and complement.
\end{cor}

Note that, in general, rational subsets of (finitely generated) groups need not be closed under intersection nor complement.


\subsection{A theorem of Herbst on rational subsets of subgroups of groups}
Let $G$ be a 
finitely generated 
group, let $U$ be a subgroup of $G$, and let $Q$ be a subset of $U$.  It is immediate from the definition of rational subset that if $Q$ is a rational subset of $U$ then $Q$ is also a rational subset of $G$.  The converse is also true,  but it is far less obvious. It was proved by Herbst in \cite{He} that, under the above assumptions, if $Q$ is a rational subset of $G$ then $Q$ is a also rational subset of $U$. 

In this subsection we will explain and give a proof of an effective version of Herbst's theorem which will be of crucial importance for us in this paper.  

Let us begin by recalling some basic facts about rational subsets and regular languages.  Let $M$ be a monoid. Just as we did for groups above, we can talk about the rational subsets of the monoid $M$. The rational subsets of $M$ are the sets that can be obtained from finite subsets using the operations of union, product and submonoid generation. If $A$ is a set of generators for $M$ then a subset of $M$ is rational if and only if it is accepted by a FSA over $A$. Here a subset $U$ of $M$ is said to be \emph{accepted by a FSA over $A$} if, with $\pi:A^* \rightarrow M$ the canonical homomorphism, there is a FSA $\mathcal{A}$ such that $L(\mathcal{A}) \pi = U$. It is a standard fact that a subset $U$ of the free monoid $A^*$ is the language of a finite state automaton if and only if $U$ can be described by a \emph{rational expression} over $A$. A rational expression for a subset $U$ of $A^*$ is a formal expression which gives a description of a way of constructing the set $U$ from finite subsets using finitely many operations of union, product and the Kleene star operation (which is submonoid generating in $A^*$). 
For example, $(ab)^* \cup (ba)^* \cup a(ba)^* \cup b(ab)^*$, is a rational expression for the regular language of all words with alternating $a$s and $b$s. Of course, different rational expressions can describe the same regular language.  It is easy to show that there is an algorithm which takes any rational expression over $A$ and constructs a FSA recognising the language that the rational expression defines, and conversely there is an algorithm which given a FSA $\mathcal{A}$ computes a rational expression defining $L(\mathcal{A})$. 
It follows that if $M$ is a monoid generated by a set $A$, then a subset $U$ of $M$ is rational if and only if there is a rational expression over $A$ which defines $U$. 
We will not be working with regular expressions much in this paper, but we shall need this notion when making reference to proofs of Herbst below. 
For a formal definition of rational expression we refer the reader to 
Section~2 of the book \cite{Pin}. 

\begin{thm}\label{thm:effective:herbst} 
Let $G$ be a finitely generated group with finite generating set $X$, and canonical homomorphism $\pi: (X \cup X^{-1})^* \rightarrow G$. 
Suppose further that $G$ has a recursively enumerable word problem. 
Let $U$ be a finitely generated subgroup of $G$ with finite generating set $Y$ and canonical homomorphism $\sigma: (Y \cup Y^{-1})^* \rightarrow U$. 
Then for every subset $Q$ of $U$ we have 
\[
Q \in \Rat(G) \Leftrightarrow Q \in \Rat(U). 
\]
Furthermore there is an algorithm which 
\begin{enumerate}
\item for any FSA $\mathcal{A}$ over $X \cup X^{-1}$ such that $L(\mathcal{A})\pi \subseteq U$  computes a FSA $\mathcal{B}$ over $Y \cup Y^{-1}$ such that $L(\mathcal{B}) \sigma = L(\mathcal{A}) \pi$, and there is an algorithm which      
\item for any FSA $\mathcal{B}$ over $Y \cup Y^{-1}$ computes a FSA $\mathcal{A}$ over $X \cup X^{-1}$ such that $L(\mathcal{A}) \pi = L(\mathcal{B}) \sigma$.   
\end{enumerate}    
\end{thm}
\begin{proof}
Throughout this proof we make use of the same notation and conventions as in \cite{He}.  
From the discussion above we know that there is an algorithm which will take a FSA $\mathcal{A}$ over $A$ recognising $Q$ and use it to compute a rational expression $\rho$ over $A$ defining $L(\mathcal{A})$.  

We now show how the argument used in the proof of \cite[Proposition 5.2]{He} can be used to prove the following claim.

\begin{claim}\label{claim:Herbst} 
 There is an algorithm which takes as input any rational expression $T$ over $X \cup X^{-1} $ such that $L(T)\pi \subseteq U$  and returns a rational expression $\ol{T}$ over $Y \cup Y^{-1} $ such that $\sigma(L(\ol{T})) = \pi(L(T)) \subseteq U$.
  \end{claim}

Here $L(R)$ denotes the language defined by a rational expression $R$. It follows from the discussion preceding the statement of the theorem that once this claim has been established then (1) will follow. We prove Claim~\ref{claim:Herbst} by induction on the starheight, denoted $\sh{T}$.  It is important to note that here by $\sh{T}$ we mean the starheight of the rational expression $T$, which might well be larger than the minimum starheight of the language $L(T)$ defined by this rational expression. That is, it is possible that there is a rational expression $T'$ with $\sh{T'} < \sh{T}$ and with $L(T') = L(T)$.         

When $\sh{T}=0$ then clearly there is an algorithm which replaces $T$ by an expression of the form  
\[
w_{1} \cup w_{2} \cup \ldots \cup w_{m}  
\]
for a finite set $\{ w_{1}, \ldots, w_{m} \} \subseteq (X \cup X^{-1} )^{*}  $ and $\pi(w_{i} ) \in U$ for $1 \leq i \leq m$. So we may assume that $T$ has this form.  Since $G$ is recursive enumerable can apply Lemma~\ref{lem:display} (see below). By Lemma~\ref{lem:display} there is an algorithm that computes words $\ol{w_{1}}, \ldots, \ol{w_{m}} \in (Y \cup Y^{-1} )^*$ such that $\sigma(\ol{w_i}) = \pi(w_i)$ for $1 \leq i \leq m$.  Hence $\sigma(\{\ol{w_{1}}, \ldots, \ol{w_{m}} \}) = \pi(\{ w_{1}, \ldots, w_{m} \})$ and this deals with this case. 

For the induction step now let us assume $\sh{T} > 0$. We can write $T = \tau_1 \cup \ldots \cup \tau_q$ where each $\tau_j$ is a rational expression of the form
\begin{equation}\label{eqn:tau:rat:expression} 
w_{1} T_{1}^{*}  w_{2} T_{2}^{*}  w_{3}  \ldots  w_{n} T_{n}^{*}
  w_{n+1}, \end{equation}
where each $w_{i} \in (X \cup X^{-1} )^{*} $ and each $T_{i} $ is a rational expression over $(X \cup X^{-1} ) $ with $\sh{T_i} < \sh{T}$ for $1 \leq i \leq n$. Consider such an expression \eqref{eqn:tau:rat:expression} corresponding to $\tau_k$.  Set 
\[
S_i = w_1 w_2 \ldots w_i T_i w_i^{-1} \ldots w_2^{-1} w_1^{-1}  
\]  
for $1 \leq i \leq n$ noting that $\sh{S_i} = \sh{T_i} < \sh{T}$ and, by \cite[Lemma~5.1]{He}, we have $\pi(L(S_i)) \subseteq U$. By induction we can suppose that our algorithm takes $S_i$ and returns a rational expression $\ol{S_i}$ over $Y \cup Y^{-1} $ such that $\pi(L(S_i)) = \sigma(L(\ol{S_i}))$. It follows from the argument given in the proof of \cite[Proposition~5.2]{He} that 
\[
\pi(L(S_1^* S_2^*  \ldots S_n^* w_1 w_2 \ldots w_{n+1} )) =
\pi(L(\tau_k)).
\]
Then we instruct our algorithm to compute 
\[
\ol{S_1}^*  \ol{S_2}^*  \ldots \ol{S_n}^*  \ol{w_1   \ldots w_2  w_{n+1}},  
\]
which is a rational expression over $Y \cup Y^{-1} $ such that 
\[
\sigma(L( 
\ol{S_1}^*  \ol{S_2}^*  \ldots \ol{S_n}^*  \ol{w_1   \ldots w_2  w_{n+1}}
)) = 
\pi(L(S_1^* S_2^*  \ldots S_n^* w_1 w_2 \ldots w_{n+1} )) =
\pi(L(\tau_k)).
\]
Repeating this for all $t_k$ in $T = \tau_1 \cup \ldots \cup \tau_q$ our algorithm computes a rational expression $\ol{T}$ over $Y \cup Y^{-1}  $ satisfying Claim~\ref{claim:Herbst}. It follows from this argument that there is a recursive algorithm satisfying the requirements of Claim~\ref{claim:Herbst}. This completes the proof of part (1).

(2)
By the comments preceding the 
statement of the theorem, 
there is an algorithm which takes $\mathcal{B}$ and computes a rational expression $R$ over $Y \cup Y^{-1} $ such that $L(R)\sigma = L(\mathcal{B})\sigma \subseteq U$.       
For each $y \in Y$ let $\hat{y} \in (X \cup X^{-1})^*$ such that $y \sigma = \hat{y} \pi$ in $U$.      
Extend this notation to $y \in Y^{-1} $ where $\widehat{y^{-1}} = \hat{y}^{-1}$, the formal inverse of the word.   
Let $\xi: (Y \cup Y^{-1})^* \rightarrow (X \cup X^{-1})^*$ be the unique homomorphism satisfying $y \xi = \hat{y}$ for all $y \in Y \cup Y^{-1}$. 
Let $S = R \xi$ be the expression obtained by replacing every word $w$ in the rational expression $w$ by the word $w \xi$. 
Then $S$ is a rational expression over $X \cup X^{-1}  $ such that $L(S)\pi = L(R)\sigma$. 
Finally apply the algorithm that takes $S$ and returns a FSA $\mathcal{A}$ over $X \cup X^{-1} $ such that $L(\mathcal{A})\pi = L(S)\pi$. 
This completes the proof.                
\end{proof}

\begin{lem}\label{lem:ultimate:herbst} 
Let $G$ be a finitely generated group with finite generating set $X$, and canonical homomorphism $\pi: (X \cup X^{-1})^* \rightarrow G$. 
Furthermore suppose that $G$ has a recursively enumerable word problem.  
Let $A, B \leq G$ be finitely generated subgroups with $\phi:A \rightarrow B$ is an isomorphism. 
Then there is an algorithm which takes any FSA $\mathcal{A}$ over $X \cup X^{-1}$ such that $L(\mathcal{A})\pi \subseteq A$  computes a FSA $\mathcal{B}$ over $X \cup X^{-1}$ such that $L(\mathcal{B}) \pi = (L(\mathcal{A}) \pi)\phi$. 
\end{lem}
\begin{proof}
Let $Y$ be a finite generating set for $A$ with canonical homomorphism $\sigma: (Y \cup Y^{-1} )^* \rightarrow A$. 
Apply the algorithm from Theorem~\ref{thm:effective:herbst}(1) to compute a FSA $\mathcal{Q}$ over $Y \cup Y^{-1} $ such that $L(\mathcal{Q})\sigma = L(\mathcal{A})\pi$.       
Since $\phi$ is an isomorphism it follows that $Y$ is a finite generating set for the group $B$ with canonical homomorphism $\sigma \phi$. 
Now apply Theorem~\ref{thm:effective:herbst}(2) to the pair $G$, $B$, with respect to the canonical homomorphism $\sigma \phi$ to compute a FSA $\mathcal{B}$ over $X \cup X^{-1} $ such that $L(\mathcal{Q}) \sigma \phi = L(\mathcal{B})\pi$.
This completes the proof since 
$
L(\mathcal{B})\pi=
L(\mathcal{Q}) \sigma \phi = 
L(\mathcal{A}) \pi \phi. 
$
\end{proof}

\begin{rmk}
Note that the hypotheses of 
Lemma~\ref{lem:ultimate:herbst}
hold in particular if $G$ has decidable word problem.   
\end{rmk}


\section{Unital and conservative factorisations}
\label{sec:factor}

Let $w\in\ol{X}^*$. A \emph{factorisation} of $w$ is a decomposition
$$
w\equiv w_1\dots w_m
$$
where 
$w_1, \dots, w_k \in \ol{X}^*$ . The words $w_i$ $(1 \leq i \leq k)$ are called the factors of this factorisation. 

Let $u_1, \dots, u_k$ be distinct words in $\ol{X}^*$ and let $w \in \ol{X}^*$. 
If $w$ belongs to the submonoid of $\ol{X}^*$ generated by $u_1,\dots,u_k$, then there
is a word $w'(x_1,\dots,x_k)$ over the alphabet $\{x_1,\dots,x_k\}$ such that 
$$
w\equiv w'(u_1,\dots,u_k). 
$$
This expression gives a factorisation of $w$ where each factor is equal to $u_j$ for some $1 \leq j \leq k$.     

A factorisation $w\equiv v_1\dots v_n$ is \emph{finer} than 
a factorisation 
$w\equiv w_1\dots w_m$ if there exist $0\leq k_1\leq\dots
\leq k_{m-1}\leq n$ such that $w_1\equiv v_1\dots v_{k_1}$, $w_m\equiv v_{k_{m-1}+1}\dots v_n$ and $w_i\equiv
v_{k_{i-1}+1}\dots v_{k_i}$ for all $1<i<m$.

We say that a factorisation $w\equiv w_1\dots w_m$ is \emph{unital} if each of its factors represents a unit of the inverse monoid $M=\Inv\pre{X}{w=1}$. In such a case, $w_1,\dots,w_m$ are called \emph{invertible pieces}. It is an easy exercise to show that there is a unique finest unital factorisation of $w$; this is a decomposition into \emph{minimal invertible pieces} $w\equiv w_1\dots w_m$, so that no proper prefix of any of the words $w_i$ represents an element of $U_M$. In fact, there is a strong connection between the minimal invertible pieces of $w$ and the group of units $U_M$.  

\begin{pro} 
Let $M=\Inv\pre{X}{w=1}$ where  $w\in \ol{X}^*$. Then the minimal invertible pieces $w_1,\dots,w_m$ of $w$ generate the group $U_M$.
\end{pro}
\begin{proof}
This follows from the argument given in the proof of \cite[Proposition 4.2]{IMM}.  
Note that, as already mentioned above, while in the statement of \cite[Proposition 4.2]{IMM} it is assumed that the words 
in the defining relators are all cyclically reduced, that assumption is not used anywhere in the proof, and the proposition 
holds with that assumption removed.  
\end{proof}

We now briefly recall the Adjan overlap algorithm (as presented in \cite{Lal}). Let $X$ be an alphabet, and assume $W_0$
is a finite set of words over $X$. Inductively, we define a sequence $W_i$, $i\geq 0$, of sets of words as follows.
Assuming that $W_k$ is determined, we let $u\in W_{k+1}$ if one of the following conditions hold:
\begin{itemize}
\item[(i)] $u\in W_k$ and $u\not\in\pref(v)\cup\suff(v)$ for all $v\in W_k\setminus\{u\}$;
\item[(ii)] there exist $v,v'\in X^*$ not both empty such that $uv,v'u\in W_k$;
\item[(iii)] there exist $v,v'\in X^*$ such that $v$ is non-empty and $uv,vv'\in W_k$;
\item[(iv)] there exist $v,v'\in X^*$ such that $v'$ is non-empty and $v'u,vv'\in W_k$.
\end{itemize}
It is fairly easy to show that the sequence of sets of words $W_k$ stabilises after finitely many steps, i.e.\ that there
exists $k_0$ such that $W_k=W_{k_0}$ for all $k\geq k_0$. Define $\Gamma=W_{k_0}$. 

This algorithm is applied to special monoid presentations by setting $W_0$ to be the set of relator words in the considered presentation; in particular, for $\Mon\pre{X}{w=1}$ we put $W_0=\{w\}$. 
The same algorithm can be applied to special inverse monoid presentation, and in particular to one-relator special inverse monoids $\Inv\pre{X}{w=1}$ by replacing $X$ by $\ol{X}$.  

It is easy to see that the words from the set $\Gamma$ generated by the Adjan algorithm all represent invertible elements of the monoid.   Adjan \cite{Adj} proved that in the case of one-relator special monoid presentations, this algorithm actually computes the decomposition of the defining relator word $w$ into minimal invertible pieces. However, this is no longer true for arbitrary special monoid presentations. 
In fact, more than this, for finitely presented special monoids the problem of computing the minimal invertible pieces is known to be undecidable. Indeed, in \cite{Otto} 
it is shown that it is undecidable whether a finitely presented special monoid is a group, and if there were an algorithm or computing the minimal invertible pieces then that algorithm could be used to decide whether a special monoid is a group by testing whether all the generators appear in at least one relator, and that the minimal invertible pieces all have size one.  

For special inverse monoids the Adjan algorithm fails to compute the minimal invertible pieces even in the one-relator case.  This is illustrated by the following example, which appeared first in print in \cite{MMSt} (see also \cite{IMM}), sometimes known as the \emph{O'Hare example} (because it was constructed by Margolis and Meakin while waiting for a connecting flight at the O'Hare International Airport, Chicago).

\begin{exa}\label{exa:ohare}
Let 
$$
M = \Inv\pre{a,b,c,d}{abcdacdadabbcdacd=1}.
$$
Applying a geometric method called \emph{Stephen's procedure} \cite{Ste} it was shown in \cite{MMSt} that
$$
w \equiv (abcd)(acd)(ad)(abbcd)(acd)
$$
is a unital factorisation of the relator word. In fact, the same methods show that any subword of the relator word representing
a unit of $M$ must begin with either $a$ or $d^{-1}$ and end with either $d$ or $a^{-1}$, so it follows that this is the
decomposition of $w$ into minimal invertible pieces. Thus it follows from
the previous proposition that $\{abcd,acd,ad,abbcd\}$ is a generating set of $U_M$ ($abbcd$ can be shown to be redundant).

On the other hand, this is not something Adjan algorithm would discover, as $\Gamma=\{abcdacdadabbcdacd\}$ in this case.
Indeed, this is the reason why the \emph{monoid} $\Mon\pre{a,b,c,d}{abcdacdadabbcdacd=1}$ has a trivial group of units,
in contrast to the inverse monoid defined by the same presentation which has an infinite group of units. 
\end{exa}

To address this, in \cite{GR} Gray and Ru\v skuc devised a new, finer pieces computing algorithm better suited for special inverse monoid presentations 
$$M=\Inv\pre{X}{w_i=1\ (i\in I)},$$ 
called the \emph{Benois algorithm} (because it relies on the Benois theorem and its consequences). Namely, let 
$$
U=\{\pref(w_i):\ i\in I\}\cup\{\pref(w_i^{-1}):\ i\in I\}.
$$
Observe that all the words in $U$ represent right invertible elements of the inverse monoid $M$. Let $V=\Mon\gen{U}$ be the submonoid of the free group $FG(X)$ generated by the words from $U$ viewed as elements of the free group i.e.\ the submonoid of $FG(X)$ generated by $\red(U)$. 
Now by Corollary \ref{cor:ben} of Benois' Theorem the free group $FG(X)$ has decidable submonoid membership problem. 
So, we can algorithmically test, for each prefix $p$ of some word $w_i$, whether $p^{-1}$ represents an element of $V$. 
If it does then $p^{-1}$ is right invertible which implies $p$ is left invertible, and hence since $p$ is also right invertible being a prefix of $w_i$, it follows that  the word $p$ represents an invertible element of the monoid $M$. Thus this algorithm gives a method for finding invertible pieces of relators. The collection of all such prefixes for which the answer is `yes' naturally gives rise to factorisations
$$
w_i\equiv w_{i,1}\dots w_{i,k_i}
$$
for all $i\in I$, such that for every prefix $p$ of $w_i$ we have
$$
p^{-1}\in V \text{ if and only if } p\equiv w_{i,1}\dots w_{i,j} \text{ for some } j\in\{1,\dots,k_i\}.
$$
It is clear from the definition that  
for all $i \in I$ the decomposition of $w_i$ into pieces computed by the Benois algorithm is unital. 
In fact, it is shown in \cite{GR} that for each word $w_i$ the factorisation computed by the Benois algorithm 
is a refinement of the decomposition computed by the Adjan algorithm.

In the same way as for special monoid presentations, there is no algorithm which takes finitely presented special inverse monoids and computes the minimal pieces of the defining relator words. In the particular case of one-relator special inverse monoids, whether the Benois algorithm computes the minimal invertible pieces is an open problem. It may be shown (see \cite{GR}) that when applied to the O'Hare monoid the Benois algorithm does compute the minimal invertible pieces of the defining relator, giving the unital factorisation of the defining relator described above in Example \ref{exa:ohare}. This example shows that there are cases where the Benois algorithm preforms strictly better than the Adjan algorithm, in the sense that it gives a decomposition which is a strict refinement of the Adjan decomposition.  

While it is not known whether the Benois algorithm computes the minimal invertible pieces for one-relator special inverse monoids, one very important central theme of the present paper will be that it is often the case that it is sufficient to find \emph{some} suitable (not necessarily minimal) unital factorisation of $w$ in order to prove that the monoid $M = \Inv\pre{A}{w=1}$ has decidable word problem. In this sense, the Benois algorithm will provide a key tool for solving the word problem for certain examples and classes of one-relator special inverse monoids.

Now we introduce another type of factorisation of a word that makes computing and handling the prefix monoid of a one-relator group presentation somewhat easier and more manageable.
Let $w \in \overline{X}^*$. Then for a factorisation
$$
w \equiv w_1\dots w_k
$$
let $P(w_1,\dots,w_k)$ denote the submonoid of $G=\Gp\pre{X}{w=1}$ generated by elements
$$
\bigcup_{i=1}^k \pref(w_i).
$$
It is quite easy to see that we always have $P_w\subseteq P(w_1,\dots,w_k)$. 
In the case that $P_w = P(w_1,\dots,w_k)$ then we say that the factorisation  
$w \equiv w_1\dots w_k$ is \emph{conservative}.     
The next result establishes a connection between unital and conservative factorisations. 

\begin{thm}\label{thm:UnitalConservative}
Let $w \in \overline{X}^{*} $. 
\begin{enumerate}
\item[(i)] Any unital factorisation of $w$ is conservative. 
\item[(ii)] If $\Inv\pre{X}{w=1}$ is $E$-unitary (in particular, if $w$ is cyclically reduced) then every conservative factorisation of $w$ is unital.     
\end{enumerate}
\end{thm}
\begin{proof}
(i) Assume $w\equiv w_1\dots w_k$ is a unital factorisation, so that each word $w_i$, $1\leq i\leq k$, represents a unit of $M=\Inv\pre{X}{w=1}$. Let $p$ be a prefix of $w_i$ for some $i$. Then
$$
p' \equiv w_1\dots w_{i-1}p
$$
is a prefix of $w$ and thus in $G=\Gp\pre{X}{w=1}$ this word represents an element of the prefix monoid $P_w$. Since $p'$ is right invertible in $M$ and $w_1 \ldots w_{i-1}$ is invertible in $M$ it follows that $p$ is right invertible in $M$. Thus by the remarks following Theorem \ref{thm:pw-wp} it follows that $p$ is equal in $M$ to a product of prefixes of $w$. Applying the natural homomorphism from $M$ to its maximal group image, it follows that $p$ represents in $G$ an element of the prefix monoid $P_w$. It follows that $P(w_1,\dots,w_k)$ must be contained in $P_w$, so in fact we have an equality of these two submonoids of $G$, confirming that the considered factorisation in conservative.

(ii) Assume that $M=\Inv\pre{X}{w=1}$ is $E$-unitary and let $w\equiv w_1\dots w_k$ be a conservative factorisation of $w$, so that we have $P(w_1,\dots,w_k)=P_w$. Our aim is to show that $w_1$ represents a unit of $M$, that is, we want to prove that $w_1\in U_M$.

Since $w_1 w_2 \ldots w_k = 1$ it is immediate that $w_1$ is a right unit of $M$. Now we need to prove that $w_1$ is a left unit which clearly is equivalent to proving that $w_1^{-1}$ is a right unit of $M$. In $G=\Gp\pre{X}{w=1}$ we have
\[
w_1^{-1} = w_2 \ldots w_k \in P(w_1, \ldots, w_k) = P_w
\]
since the factorisation is conservative. Thus in $G$ we have 
$$
w_1^{-1}=p_1\dots p_m
$$
for some $p_i \in\pref(w)$ for all $1 \leq i \leq m$. Since $M$ is $E$-unitary, it follows from Lemma \ref{lem:sim} that we have $w_1^{-1}\sim p_1\dots p_m$ in $M$. This implies that $w_1p_1\dots p_m\in E(M)$. Since the only right invertible idempotent in an inverse monoid is the identity element it follows that $w_1 p_1 \ldots p_m = 1$ in $M$. From this we deduce that $w_1 p_1 \ldots p_m w_1 = w_1$ and $p_1 \ldots p_m w_1 p_1 \ldots p_m= p_1 \ldots p_m$ and hence $w_1^{-1} = p_1 \ldots p_m$ in $M$. We conclude that $w_1^{-1}$ represents an invertible element of $M$. This concludes the proof that $w_1$ represents an invertible element of the monoid $M$.  

Consequently, 
$$
w_2\dots w_kw_1 = w_1^{-1}(w_1\dots w_k)w_1=w_1^{-1}w_1=1,
$$
so now we can argue, by repeating the previous argument, that $w_2\in U_M$. In this say we can prove that each factor represents an invertible element of the monoid. We conclude that the considered factorisation is unital.
\end{proof}

\begin{cor}
Let $w\in\ol{X}^*$.
Then the Benois algorithm applied to $w$ computes a conservative 
factorisation of $w$.
\end{cor}

This importance of this corollary will become in Section~\ref{sec:appl1} where  there are theorems whose hypotheses are that the defining relator admits a conservative factorisation satisfying certain properties. 
This corollary will help us verify that these hypotheses do hold in particular concrete examples.


\section{Deciding membership in submonoids of amalgamated free products}

\label{sec:amalg}

As explained in the introduction, our aim in this section is to give two general decidability results concerning the membership problem for certain submonoids of free amalgamated products $B*_A C$ of finitely generated groups. Then in the next section these results will be applied to the prefix membership problem for certain one-relator groups.

Let us now fix the notation and conventions that will be in place throughout this section. We refer the reader to \cite{LSch} for more background and proofs of standard results. Throughout this section we use $B*_A C$ to denote the amalgamated free product of two groups $B$ and $C$ over a group $A$. We shall always assume that all three of these groups are finitely generated by, respectively, the finite sets $X$, $Y$ and $Z$. We formalise this by fixing canonical homomorphisms $\pi:\ol{X}^*\to B$, $\theta:\ol{Y}^*\to C$ and $\xi:\ol{Z}\to A$.

In addition, we have two injective homomorphisms $f:A\to B$ and $g:A\to C$ and $B \ast_A C$ is the corresponding pushout in the category of groups. 
Then the amalgamated free product $B \ast_A C$ is defined by the presentation obtained by taking the disjoint union of presentations for $B$ and $C$ together with additional defining relations $zf = zg$ for all $z \in Z$. To be more precise, let $Z=\{z_1,\dots,z_m\}$. Then there are words $\alpha_i\in\ol{X}^*$ and $\beta_i\in\ol{Y}^*$, $1\leq i\leq m$, such that the mappings $f: Z \rightarrow \ol{X}^*$ and $g:Z \rightarrow \ol{Y}^*$ defined by $z_if = \alpha_i$ and $z_ig = \beta_i$ induce injective homomorphisms  $f:A\to B$ and $g:A\to C$. Note that $f$ and $g$ are used to denote both mappings from $A$ into $B$ and $C$, respectively, and also to define mappings on words $f: \ol{Z}^* \rightarrow \ol{X}^*$ and $g:\ol{Z}^* \rightarrow \ol{Y}^*$. Consequently,
$$
wf \equiv w(\alpha_1,\dots,\alpha_m)\text{ and }wg \equiv w(\beta_1,\dots,\beta_m)
$$ 
for any $w\in \ol{Z}^*$. 
Recall that $w(\alpha_1, \ldots, \alpha_m)$ is the word in $\ol{X^*}$ obtained by replacing $z_i$ by $\alpha_i$ for each letter $z_i\in Z$ in the word $w$.
Using this notation, we may speak about the membership problem for the subgroup $A$ in $B$ and $C$ respectively: 
the algorithmic question is whether there is an algorithm which takes as input any word  $u$ over $\ol{X}$ (resp.\ $\ol{Y}$)  and decides whether or not there  exists a word $w\in\ol{Z}^*$ such that $u=w(\alpha_1,\dots, \alpha_m)$ holds in $B$ (resp.\ $u=w(\beta_1,\dots,\beta_m)$ holds in $C$).
We will often identify $A$ with its image in $B \ast_A C$. In this way, each of $A$, $B$ and $C$ is viewed as a subset of the amalgamated free product $B \ast_A C$ and $B \cap C = A$. 
So, for any $b \in B$ by saying that $b$ \emph{belongs to} $A$ we mean that $b \in Af$, and analogously we talk about an element $c \in C$ belonging to $A$.  

We are now in a position to state the two main general results of this section. 

\begin{theoremletter}\label{thm:amal}
Let $G=B*_A C$, where $A,B,C$ are finitely generated groups such that both $B,C$ have decidable word problems, and the membership problem for $A$ in both $B$ and $C$ is decidable. Let $M$ be a submonoid of $G$ such that the following conditions hold:
\begin{itemize}
\item[(i)] $A\subseteq M$;
\item[(ii)] both $M\cap B$ and $M\cap C$ are finitely generated and 
            $$M=\Mon\gen{(M\cap B)\cup(M\cap C)};$$
\item[(iii)] the membership problem for $M\cap B$ in $B$ is decidable;
\item[(iv)] the membership problem for $M\cap C$ in $C$ is decidable.
\end{itemize}
Then the membership problem for $M$ in $G$ is decidable.
\end{theoremletter}

\begin{definition}(Closed for rational intersections)
Let $G$ be a finitely generated group, generated by a finite set $\Omega$ with canonical homomorphism $\tau: \ol \Omega \rightarrow G$, and let $H$ be a finitely generated subgroup of $G$. We say that \emph{$H$ in $G$ is closed for rational intersections} if $R \cap H \in \Rat(G)$ for all $R \in \Rat(G)$. We say that \emph{$H$ in $G$ is effectively closed for rational intersections} if it is closed for rational intersections and moreover that there is an algorithm which given a FSA $\mathcal{A}$ over $\ol{\Omega}$ computes a FSA $\mathcal{A}_H$ over $\ol{\Omega}$ such that $L(\mathcal{A}_H)\tau = (L(\mathcal{A})\tau) \cap H$. 
\end{definition}

\begin{theoremletter}\label{thm:amal51}
Let $G=B*_A C$, where $A,B,C$ are finitely generated groups. Let $M$ be a submonoid of $G$ such that both $M\cap B$ and $M\cap C$
are finitely generated and $M=\Mon\gen{(M\cap B)\cup(M\cap C)}$. Assume further that the following conditions hold:
\begin{itemize}
\item[(i)] $B$ and $C$ have decidable rational subset membership problems;
\item[(ii)] $A \leq B$ is \EC;  
\item[(iii)] $A \leq C$ is \EC.  
\end{itemize}
Then the membership problem for $M$ in $G$ is decidable.
\end{theoremletter}

Comparing these two theorems,  Theorem~\ref{thm:amal} shows that  the membership problem for   suitably nice submonoids in relatively general amalgamated free products  is decidable, while  Theorem~\ref{thm:amal51} shows that under stronger assumptions on the groups of the amalgamated free product, we get a much broader family of submonoids  in which we can decide membership.

As is well known, see e.g.  \cite[pp.\ 186-187]{LSch}, each element $g\in G=B*_A C$ can be written as 
$$
g = b_1c_1\dots b_nc_n
$$
with $b_i\in B$ and $c_i\in C$ for all $1\leq i\leq n$. 
The above representation is said to be in \emph{reduced form} if
\begin{itemize}
\item If $n>1$ then $b_i$ does not belong to $A$ for all $i \neq 1$, 
$b_1$ is either equal to $1$ or else does not belong to $A$, 
$c_i$ does not belong to $A$ for all $i \neq n$, and 
$c_n$ is either equal to $1$ or else does not belong to $A$.  
\item If $n=1$ then if both $b_1$ and $c_1$ belong to $A$ then exactly one of $b_1=1$ or $c_1=1$ holds.   
\end{itemize}
Moreover, a word  $w\equiv u_1v_1\dots u_kv_k$  where $u_i\in\ol{X}^*$ and $v_i\in\ol{Y}^*$, $1\leq i\leq k$,  is said to be \emph{a word in reduced form} if and only if $(u_1 \pi) (v_1 \theta) \dots (u_k \pi)(v_k \theta)$ is a reduced form.
The following result is standard and can be proved e.g.\ by applying \cite[Theorem IV.2.6]{LSch}. 

\begin{lem}\label{lem:rf-amal}
An equality of two reduced forms
$$
b_1c_1\dots b_nc_n = p_1q_1\dots p_kq_k 
$$
holds in $G=B*_A C$ if and only if the following conditions are satisfied:
\begin{itemize}
\item[(i)] $n=k$;
\item[(ii)] there exist $1=a_0,a_1,\dots,a_{2n-1},a_{2n}=1\in A$ such that for all $1\leq i\leq n$ we have
$$
p_i = a_{2i-2}^{-1}b_ia_{2i-1}\text{ and }q_i = a_{2i-1}^{-1}c_ia_{2i}.
$$
\end{itemize} 
\end{lem} 

Before embarking on the proofs of Theorems~\ref{thm:amal} and \ref{thm:amal51}, we first collect several useful lemmas.  

\begin{lem}\label{lem:reduce}
Let $G=B*_A C$ and let 
$M$ be a submonoid of $G$ such that
            $$M=\Mon\gen{(M\cap B)\cup(M\cap C)}.$$
Then every element $g \in M$ can be written in reduced form 
\[
g = p_1q_1\dots p_nq_n
\]
where $p_i \in M \cap B$ and $q_i \in M \cap C$ for $1 \leq i \leq n$.
\end{lem}

\begin{proof}
By assumption we can write
$$
g = r_1s_1\dots r_ks_k
$$
for some $r_i\in M\cap B$ and $s_i\in M\cap C$, $1\leq i\leq k$. If this is in reduced form we are done. Otherwise, there is some term, say  $r_i$, with $r_i \in A$. Since $r_i\in (M\cap B)\cap A = M\cap A\subseteq M\cap C$, we have $s_{i-1}'=s_{i-1}r_is_i\in M\cap C$, so upon relabelling $r_j'=r_{j+1}$, $s_j'=s_{j+1}$ for $i\leq j<k$ we get
$$
g = r_1s_1\dots r_{i-1}s_{i-1}'r_i's_i'\dots r_{k-1}'s_{k-1}',
$$
a shorter alternating product of elements of $M\cap B$ and $M\cap C$. 
A similar argument applies if there is a term $s_i$ with $s_i \in A$. 
At the end of this process we arrive at a reduced form $p_1q_1\dots p_nq_n$ 
for $g$ whose terms alternatively belong to $M\cap B$ and $M\cap C$, completing the proof of the lemma. 
\end{proof}

A crucial algorithmic aspect is settled by the following observation. 
The proof is routine and so it is omitted. 

\begin{lem}\label{lem:display}
Let $G$ be a group finitely generated by $\Omega$ and suppose that $G$ has a recursively enumerable word problem. Let $w_1,\dots,w_n\in\ol{\Omega}^*$ and 
\[ 
H=\Gp\gen{w_1,\dots,w_k} \leq G.
\] 
Suppose that the membership problem for $H$ in $G$ is decidable. Then there exists an algorithm which, given any word
$w\in\ol{\Omega}^*$ for which the algorithm for testing membership in $H$ returns `yes', outputs a word 
$u(t_1,\dots,t_k)\in \ol{T}^*$, where $T=\{t_1,\dots,t_k\}$, such that
$$
w = u(w_1,\dots,w_k)
$$ 
holds in $G$. 
\end{lem}

The following key lemma identifies conditions under which  the process in Lemma \ref{lem:reduce} can be performed algorithmically. 

\begin{lem}
\label{lem:eff-reduce}
Let $G=B*_A C$ and assume that the following conditions hold:
\begin{itemize}
\item $B$ and $C$ both have recursively enumerable word problem;
\item the membership problem for $A$ in $B$ is decidable; and 
\item the membership problem for $A$ in $C$ is decidable.
\end{itemize}
Then there is an algorithm which takes as input any word $w\equiv u_1v_1\dots u_kv_k$  where $u_i\in\ol{X}^*$ and $v_i\in\ol{Y}^*$, $1\leq i\leq k$,  and returns a word $p_1q_1\dots p_nq_n$ in reduced form where $p_i\in\ol{X}^*$ and $q_i\in\ol{Y}^*$, $1\leq i\leq k$,  such that $w = p_1q_1\dots p_nq_n$ holds in $G$. 
\end{lem}
\begin{proof}
It follows by the assumptions that there is algorithm which decides for each of the terms $u_i$ and $v_j$ whether or not that term represents an element of $A$. 
If none of the terms represents an element of $A$ then 
$u_1v_1 \dots u_kv_k$ is a word in reduced form and the algorithm terminates and outputs this word. 
Otherwise, suppose that some $u_i$ or $v_j$ does represent an element of $A$. 
Let $u_i$ be the first term that the algorithm detects as belonging to $A$. 
Since $B$ has 
recursively enumerable word problem
and the membership problem for $A$ within $B$ is decidable,  
we can apply Lemma~\ref{lem:display}. 
This tells us that there is an algorithm which takes any such $u_i$ as input and  
returns a word $u'\in\ol{Z}^*$ such that we have  
$$
u_i = u'(\alpha_1,\dots,\alpha_m),
$$
in $B$,  where $\alpha_i = z_i f$ for $1 \leq i \leq m$. 
The algorithm then computes the word 
$v_{i-1}u'(\beta_1,\dots,\beta_m) v_i\in\ol{Y}^*$. 
Let $v_{i-1}’$ denote this word. 
Then, since $u'(\alpha_1,\dots,\alpha_m) =   u'(\beta_1,\dots,\beta_m)$ in $G$, it follows that 
$$
u_1v_1\dots u_{i-1}v_{i-1}'u_i'v_i'\dots u_{k-1}'v_{k-1}',
$$
is equal to $w$ in $G$, 
where $u_j'=u_{j+1}$ and $v_j'=v_{j+1}$ for $i\leq j<k$. 
This word has strictly fewer terms than the input word $w$. 
If $v_i$ is the first term that the algorithm detects as belonging to $A$ then a similar argument applies, working within $C$.  
In this way, in a finite number of steps the algorithm eventually terminates outputting a word in reduced form.
\end{proof}

The following straightforward consequence of Theorem~\ref{thm:effective:herbst} will be important for the proof of Theorem~\ref{thm:amal51}.  

\begin{lem}\label{lem:left-to-right}
Let $G=B*_A C$  where $A,B,C$ are finitely generated groups with finite generating sets $X$, $Y$, and $Z$ respectively, and canonical homomorphisms  $\pi: \ol{X}^* \rightarrow B$,   $\theta: \ol{Y}^* \rightarrow C$,   $\xi: \ol{Z}^* \rightarrow A$. 
Then we have the following.  
\begin{enumerate}
    \item[(i)] There is an algorithm which takes any FSA $\mathcal{P}$ over $\ol{X}$ such that $[L(\mathcal{P})] \pi \subseteq A$ as input and returns a FSA $\mathcal{P}'$ over $\ol{Z}$, and a FSA $\mathcal{P}''$ over $\ol{Y}$, such that  
    \[
    [L(\mathcal{P''})] \theta = [L(\mathcal{P'})] \xi = [L(\mathcal{P})] \pi. 
    \] 
    \item[(ii)] There is an algorithm which takes any FSA $\mathcal{Q}$ over $\ol{Y}$ such that $[L(\mathcal{Q})] \theta \subseteq A$ as input and returns a FSA $\mathcal{Q}'$ over $\ol{Z}$, and a FSA $\mathcal{Q}''$ over $\ol{X}$, such that  
    \[
    [L(\mathcal{Q''})] \pi= [L(\mathcal{Q'})] \xi = [L(\mathcal{Q})] \theta. 
    \]
\end{enumerate}
\end{lem}

\subsection{Proof of Theorem~\ref{thm:amal}}

A key ingredient in our proof is summarised in the following auxiliary result, which strengthens Lemma \ref{lem:reduce} under the stronger conditions of Theorem \ref{thm:amal}.

\begin{pro}\label{pro:g-in-M}
Assuming all the notation and conditions from Theorem \ref{thm:amal}, let
$$
g = b_1c_1\dots b_nc_n
$$
be an element of $G=B*_A C$ written in reduced form. Then $g\in M$ if and only if $b_i\in M\cap B$ and
$c_i\in M\cap C$ for all $1\leq i\leq n$. 
\end{pro}
\begin{proof}
($\Rightarrow$) 
By Lemma \ref{lem:reduce}, we can write 
\[
g = p_1q_1\dots p_mq_m
\]
in reduced form such that $p_i \in M \cap B$ and $q_i \in M \cap C$ for $1 \leq i \leq m$.
Now, by Lemma~\ref{lem:rf-amal} we have $m=n$ and
$$
b_i=a_{2i-2}^{-1}p_ia_{2i-1}\quad \mbox{and}\quad c_i=a_{2i-1}^{-1}q_ia_{2i}
$$
for some $a_j\in A$, $0\leq j\leq 2n$, $1\leq i\leq n$. But this implies $b_i\in A(M\cap B)A\subseteq M\cap B$ 
and  $c_i\in A(M\cap C)A\subseteq M\cap C$, as $A\subseteq (M\cap B)\cap (M\cap C)$ by condition (i) in Theorem~\ref{thm:amal}.

($\Leftarrow$) This is trivial, as $M\cap B$ and $M\cap C$ are both subsets of $M$.
\end{proof}

\begin{proof}[Proof of Theorem \ref{thm:amal}]
To prove the theorem we must show that  there is an algorithm which takes any word $w$ from $(\ol{X} \cup \ol{Y})^*$ as input and decides whether or not the word represents an element of the submonoid $M$. The hypotheses of   Lemma~\ref{lem:eff-reduce} are satisfied since by the assumptions $B$ and $C$ both have decidable word problem, and the membership problem for $A$ in each of $B$ and $C$ is decidable.  Applying this lemma  we conclude that  there is an algorithm that given such a word $w$ returns a word $p_1q_1\dots p_nq_n$ in reduced form where $p_i\in\ol{X}^*$ and $q_i\in\ol{Y}^*$, $1\leq i\leq k$, such that $w = p_1q_1\dots p_nq_n$ holds in $G$.   It follows from  Proposition \ref{pro:g-in-M} that $w \in M$ if and only if $p_i\in M\cap B$ and $q_i\in M\cap C$ for all $1\leq i\leq n$, which can be decided by  conditions (iii) and (iv). 
\end{proof}

With applications in mind, it is worthwhile to record a consequence of Theorem \ref{thm:amal} for free products of groups, arising from the case when the amalgamated subgroup $A$ is trivial.

\begin{cor}\label{cor:free-pr}
Let $G=B*C$, where $B,C$ are finitely generated groups such that both $B,C$ have decidable word problems. Let $M$ be a submonoid of $G$ such that the following conditions hold:
\begin{itemize}
\item[(i)] both $M\cap B$ and $M\cap C$ are finitely generated and 
            $$M=\Mon\gen{(M\cap B)\cup(M\cap C)};$$
\item[(ii)] the membership problem for $M\cap B$ in $B$ is decidable;
\item[(iii)] the membership problem for $M\cap C$ in $C$ is decidable.
\end{itemize}
Then the membership problem for $M$ in $G$ is decidable.
\end{cor}

\subsection{Proof of Theorem~\ref{thm:amal51}}

The following result which gives necessary and sufficient conditions for an element in reduced form to belong to $M$, will be essential for the proof of Theorem~\ref{thm:amal51}.

\begin{pro}\label{pro:amal55}
Let $G=B*_A C$ and let $M$ be a submonoid of $G$ such that
            $$M=\Mon\gen{(M\cap B)\cup(M\cap C)}.$$
Let
$$
g = b_1c_1\dots b_nc_n
$$
be an element of $G=B*_A C$ written in reduced form. 
For $i \in \{0, \ldots, 2n-1 \}$ define 
$Q_i = Q_i(b_1,c_1, \ldots, b_n, c_n)$ in the following way: 
\begin{align*}
Q_0 &= \{1\},\\
Q_{2k-1} &= (M\cap B)^{-1}Q_{2k-2}b_k \cap A, \; \mbox{for $1 \leq k \leq n$,}\\
Q_{2k} &= (M\cap C)^{-1}Q_{2k-1}c_k \cap A, \; \mbox{for $1 \leq k \leq n-1$.}
\end{align*}
Then $g\in M$ if and only if 
$$
c_n\in Q_{2n-1}^{-1}(M\cap C).
$$
\end{pro}

\begin{proof}
($\Rightarrow$) Assume that $g\in M$. 
Then by assumption the hypotheses of 
Lemma~\ref{lem:reduce}
are satisfied and thus
the element
$g$ can be written in reduced form
$$
g=p_1q_1\dots p_nq_n
$$
such that $p_i\in M\cap B$ and $q_i\in M\cap C$, $1\leq i\leq n$. 
Now, by Lemma~\ref{lem:rf-amal}, we must have
\begin{align*}
b_1 &= p_1a_1 & c_1 &= a_1^{-1}q_1a_2,\\
b_2 &= a_2^{-1}p_2a_3, & c_2 &= a_3^{-1}q_2a_4,\\
&\vdots & & \vdots \\
b_n &= a_{2n-2}^{-1}p_na_{2n-1}, & c_n &= a_{2n-1}^{-1}q_n,
\end{align*}
for some $a_1,\dots,a_{2n-1}\in A$. 
Solving alternatively for $a$'s with odd indices from the first and with even ones from the
second column of equations, we obtain
\begin{align*}
a_1 &= p_1^{-1}b_1 \in (M\cap B)^{-1}b_1 \cap A = Q_1,\\
a_2 &= q_1^{-1}a_1c_1 \in (M\cap C)^{-1}Q_1c_1 \cap A = Q_2,\\
a_3 &= p_2^{-1}a_2b_2 \in (M\cap B)^{-1}Q_2b_2 \cap A = Q_3,\\
&\vdots \\
a_{2n-1} &= p_n^{-1}a_{2n-2}b_n \in (M\cap B)^{-1}Q_{2n-2}b_n \cap A = Q_{2n-1}.
\end{align*}
Therefore, from the last equation of the second column we conclude that
$$
c_n = a_{2n-1}^{-1}q_n \in Q_{2n-1}^{-1}(M\cap C),
$$
as required.

($\Leftarrow$) Assume that $g=b_1c_1\dots b_nc_n$ is such that $c_n\in Q_{2n-1}^{-1}(M\cap C)$. Then $c_n=
\xi_{2n-1}^{-1}\gamma_n$ for some $\xi_{2n-1}\in Q_{2n-1}$ and $\gamma_n\in M\cap C$. The fact that $\xi_{2n-1}\in Q_{2n-1}=
(M\cap B)^{-1}Q_{2n-2}b_n \cap A$ implies that we can write $\xi_{2n-1}=\beta_n^{-1}\xi_{2n-2}b_n$ for some $\beta_n\in M\cap B$
and $\xi_{2n-2}\in Q_{2n-2}$. 
Continuing this process 
yields elements $\beta_{n-1},\dots,\beta_1\in M\cap B$,
$\gamma_{n-1},\dots,\gamma_1\in M\cap C$ and $\xi_i\in Q_i$, $0\leq i\leq 2n-1$, such that 
$$
\xi_{2j-1}=\beta_j^{-1}\xi_{2j-2}b_j
$$
for $1\leq j\leq n$ (where $\xi_0=1$), and 
$$
\xi_{2j}=\gamma_j^{-1}\xi_{2j-1}c_j
$$
for $1\leq j\leq n-1$. 
Solving each of these equations for 
$b_j$ and $c_j$, 
substituting 
into the reduced form of $g$, and cancelling the $\xi_k$'s gives 
$$
b_1c_1\dots b_nc_n = \beta_1\gamma_1\dots \beta_n\gamma_n,
$$
which belongs to $M$ since $\beta_j, \gamma_j \in M$ for all $1 \leq j \leq n$. 
\end{proof}

\begin{lem}\label{lem:Qrat}
Under the assumptions of Theorem~\ref{thm:amal51},
in the statement of Proposition~\ref{pro:amal55}  
every set $Q_i$ ($0 \leq i \leq 2n-1$) 
is a rational subset of $A$. 
\end{lem}
\begin{proof}
By assumption both $M \cap B$ and $M \cap C$ are finitely generated submonoids of 
$B$ and $C$, respectively. 
Hence $(M \cap B)^{-1}$ is a rational subset of $B$, and 
$(M \cap C)^{-1}$ is a rational subset of $C$. 
The lemma follows from this combined with conditions (ii) and (iii) in the statement of Theorem~\ref{thm:amal51} and the definition of $Q_i$. 
\end{proof}

It is very important to note that the sequence of rational subsets $Q_i$ given in   Proposition~\ref{pro:amal55} depends on the reduced form $b_1 c_1 \ldots b_n c_n$.  

\begin{proof}[Proof of Theorem \ref{thm:amal51}]
Similarly to the proof of Theorem \ref{thm:amal}, to prove the theorem we must show that 
there is an algorithm which takes any word $w$ from $(\ol{X} \cup \ol{Y})^*$ as input and decides whether or not the word represents an element of the submonoid $M$. 
By assumption (i) it follows that the groups $B$ and $C$ both have decidable subgroup membership problem, and in particular both have decidable word problem. 
Condition (i) also implies that the membership problem for $A$ within $B$ is decidable, and for $A$ within $C$ is decidable. 
Hence, the hypotheses of Lemma~\ref{lem:eff-reduce} are satisfied. 
Applying this lemma  we conclude that 
there is an algorithm that given any such word $w$ returns a word  $p_1q_1\dots p_nq_n$ in reduced form where $p_i\in\ol{X}^*$ and $q_i\in\ol{Y}^*$, $1\leq i\leq k$,  such that $w = p_1q_1\dots p_nq_n$ holds in $G$.  

Set $b_i = p_i \pi$ and $c_i = q_i \theta$ for $1 \leq i \leq n$, and let $g = b_1c_1 \ldots b_nc_n$ noting that this is a reduced form for the element $g$. 
For each $i \in \{0, \ldots, 2n-1 \}$ let $Q_i = Q_i(b_1,c_1,\ldots,b_n,c_n)$ be defined as in the statement of Proposition~\ref{pro:amal55}.
Then by Lemma~\ref{lem:Qrat} each of these sets $Q_i$ is a rational subset of $A$, and therefore also a rational subset of both $B$ and $C$.  

\begin{claim}
There exists an algorithm which for each $i \in \{0, \ldots, 2n-1 \}$ computes 
\begin{itemize}
    \item a finite state automaton $\mathcal{A}_i$ over $\ol{Z}$ with 
    $[L(\mathcal{A}_i)]\xi = Q_i$, 
    \item a finite state automaton $\mathcal{B}_i$ over $\ol{X}$ with 
    $[L(\mathcal{B}_i)]\pi = Q_i$, and
  \item a finite state automaton $\mathcal{C}_i$ over $\ol{Y}$ with 
    $[L(\mathcal{C}_i)]\theta = Q_i$.        
\end{itemize}
\end{claim}
\begin{proof}[Proof of claim.] 
The algorithm iteratively constructs the triples $(\mathcal{A}_i, \mathcal{B}_i, \mathcal{C}_i)$ in the following way. When $i=0$ we have $Q_i = \{ 1 \}$ and it is clear that an appropriate triple $(\mathcal{A}_0, \mathcal{B}_0, \mathcal{C}_0)$ 
can be computed e.g. by taking automata that accept only the empty word in each case.  
Now consider a typical stage $i$ with $i >0$. There are two cases depending on the parity of $i$. 

First suppose that $i$ is odd, and write $i=2k-1$. Then by definition $Q_i=(M\cap B)^{-1}Q_{i-1}b_k \cap A$.  Since $M \cap B$ is assumed to be finitely generated, there is a fixed FSA (depending only on $M$) over $\ol{X}$, which we denote by $\mathcal{B}$, satisfying $[L(\mathcal{B})]\pi = (M \cap B)^{-1}$.  Using $\mathcal{B}$ and $\mathcal{B}_{i-1}$ the algorithm then produces, in the obvious way, a FSA $\mathcal{B}^{(i)}$ over $\ol{X}$   such that  $[L(\mathcal{B}^{(i)})] \pi = (M\cap B)^{-1}Q_{i-1}b_k$.  The algorithm given by assumption (ii), in the statement of the theorem,  is then applied to the automaton $\mathcal{B}^{(i)}$ which yields an automaton  $\mathcal{B}_i$ over $\ol{X}$  satisfying 
\[
[L(\mathcal{B}_i)]\pi =  (M\cap B)^{-1}Q_{i-1}b_k \cap A = Q_i.
\]
The algorithm then calls as a subroutine the algorithm given in Lemma~\ref{lem:left-to-right} to compute automata $\mathcal{A}_i$ and $\mathcal{C}_i$  the properties given in the statement of the claim. 

If $i$ is even the procedure is analogous but with the roles of $B$ and $C$ interchanged.  
\end{proof}

To complete the proof, by Proposition~\ref{pro:amal55}, we have $g \in M$ if and only if   $c_n \in Q_{2n-1}^{-1}(M\cap C)$.  Using the automata $\mathcal{C}_{2n-1}$ and $\mathcal{C}$, the algorithm produces, in the obvious way,  a FSA $\mathcal{C}(w)$ over $\ol{Y}$ such that $[L(\mathcal{C}(w))]\theta = Q_{2n-1}^{-1}(M\cap C)$. 

Therefore, in summary we have shown that there is an algorithm which given any word $w \in (\ol{X} \cup \ol{Y})^*$ computes 
a word $p_1 q_1 \ldots p_n q_n$ in reduced form, equal to $w$ in $G$, and also computes 
a FSA $\mathcal{C}(w)$ over $\ol{Y}$ such that $w$ represents an element of $M$ if and only if $q_n \theta \in [L(\mathcal{C}(w))]\theta$. This is decidable by condition (i) of the theorem.  
\end{proof}

\begin{cor}
Let $X$ and $Y$ be finite alphabets, and let $G=FG(X) *_A FG(Y)$ such that $A$ is finitely generated. let $M$ be a submonoid
of $G$ such that both $M\cap FG(X)$ and $M\cap FG(Y)$ are finitely generated and 
$$M=\Mon\gen{(M\cap FG(X))\cup(M\cap FG(Y))}.$$
Then the membership problem of $M$ within $G$ is decidable.
\end{cor}

\section{Applications of amalgamated free product results to the prefix membership problem}
\label{sec:appl1}

In this section we present several applications of the general results from the previous section to the prefix membership problem for one-relator groups and the word problem for one-relator inverse monoids.   

We fix some terminology that will be in place throughout the section. Let $v \in \ol{X}^*$ and let $x \in X$. 
We say that the letter $x$ \emph{appears} in the word $v$ if either $v \equiv v_1 x v_2$ or $v \equiv v_1 x^{-1} v_2$ for some words $v_1, v_2 \in \ol{X}^*$. 
Given two words $w_1, w_2 \in \ol{X}^*$ we say that $w_1$ and $w_2$ have no letters in common if there is no $x \in X$ which appears both in $w_1$ and in $w_2$.  
Furthermore, let $z \in \ol{X}^*$ and let $x \in \ol{X}$. We say that $z$ \emph{contains} $x$ if $z \equiv z_1 x z_2$ for some $z_1, z_2 \in \ol{X}^*$.

\subsection{Unique marker letter theorem}

\begin{thm}\label{thm:marker}
Let $G=\Gp\pre{X}{w=1}$ and let $u = u(y_1, \ldots, y_k) \in\ol{Y}^*$, with $Y=\{y_1,\dots,y_k\}$, be such that the decomposition $w\equiv u(w_1,\dots,w_k)$ determines a conservative factorisation of $w$, where $w_1, \ldots, w_k \in \overline{X}^*$. 
Suppose that for all $i\in\{1,\dots,k\}$ there is a letter $x_i\in X$ that appears exactly once in $w_i$ and does not appear in any $w_j$ for $j\neq i$. Then the group $G=\Gp\pre{X}{w=1}$ has decidable prefix membership problem. 

Consequently, if  the above conditions are satisfied, and the one-relator inverse monoid  $\Inv\pre{X}{w=1}$  is $E$-unitary (in particular, if $w$ is cyclically reduced) then $\Inv\pre{X}{w=1}$  has decidable word problem. 
\end{thm}

\begin{proof}
Denote $X_0=\{x_1,\dots,x_k\}$ and set $X_1=X\setminus X_0$. 
So for each $1 \leq i \leq k$ the letter $x_i$ appears exactly once in the word $w_i$ (either as $x_i$ or $x_i^{-1}$) and $x_i$ does not appear in any of the words $w_j$ with $j \neq i$. 
Therefore, for all $1\leq i\leq k$, we can write 
$$
w_i\equiv p_i x_i^{\eps_i} q_i
$$
where $\eps_i\in\{1,-1\}$ and $p_i,q_i\in\ol{X_1}^*$. Let us now apply Tietze transformations
to the initial presentation of $G$ by introducing new letters $Z=\{z_1,\dots,z_k\}$ with the aim of replacing
the factors $w_1,\dots,w_k$.
The conditions on the words $w_i$ then allow us to apply further Tietze transformations showing that the generators $x_i$ are redundant and thus can be eliminated giving a presentation just in terms of the generators $X_1 \cup Z$. 
This gives 
\begin{align*}
G &=
\Gp\pre{X_0\cup X_1\cup Z}{z_i=w_i\ (1\leq i\leq k),\ u(w_1,\dots,w_k)=1}\\
&= \Gp\pre{X_0\cup X_1\cup Z}{z_i=w_i\ (1\leq i\leq k),\ u(z_1,\dots,z_k)=1}\\
  &= \Gp\pre{X_0\cup X_1\cup Z}{x_i^{\eps_i}=p_i^{-1}z_iq_i^{-1}(1\leq i\leq k),\ u(z_1,\dots,z_k)=1}\\
	&= \Gp\pre{X_1\cup Z}{u(z_1,\dots,z_k)=1}.
\end{align*}
Therefore, $G=FG(X_1)*H$, where $H=\Gp\pre{Z}{u(z_1,\dots,z_k)=1}$ is a one-relator group.

We now turn to considering the prefix monoid $P_w = \Mon\gen{\pref{w}} \leq G$. 

Without loss of generality, we may suppose that the letters of the alphabet $Y = \{y_1, \ldots, y_k\}$ are ordered in such a way that the following conditions hold:
\begin{itemize}
    \item $y_1, \ldots, y_r$ appear in $u$ while none of $y_{r+1}, \ldots y_k$ appears in $u$;  
    \item $y_s^{-1}, \ldots, y_k^{-1}$ appear in $u$ while none of $y_1^{-1}, \ldots, y_{s-1}^{-1}$ appears in $u$, 
\end{itemize}
where $r \in \{0, \ldots, k\}$ and $s \in \{1, \ldots, k+1\}$ and $s \leq r+1$. The condition $s \leq r+1$ comes from the fact that all of the letters $y_1, \ldots, y_k$ appear in the word $u = u(y_1, \ldots, y_k)$ either as $y_j$ or $y_j^{-1}$.

Since the given factorisation is assumed to be conservative,
$$
P_w = \Mon\gen{\pref(w_1)\cup\dots\cup\pref(w_r)\cup\pref(w_s^{-1})\cup\dots\cup\pref(w_k^{-1})}.
$$
Clearly, in the group $G$ we have the following equalities of sets
\begin{align*}
\pref(w_i) &= \pref(p_i) \cup p_i x_i^{\eps_i}\cdot\pref(q_i)\\
           &= \pref(p_i) \cup z_iq_i^{-1}\cdot\pref(q_i)\\
	       &= \pref(p_i) \cup z_i\cdot\pref(q_i^{-1}).
\end{align*}
In particular, $z_i$ is among the elements represented by prefixes of $w_i$. Similarly,
$$
\pref(w_i^{-1}) = \pref(q_i^{-1})\cup z_i^{-1}\cdot\suff(p_i^{-1})^{-1} = 
                  \pref(q_i^{-1})\cup z_i^{-1}\cdot\pref(p_i), 
$$
which includes the element $z_i^{-1}$. Thus $P_w$ is equal to the submonoid of $FG(X_1)*H$ generated by the set
$$
\{z_1,\dots,z_r,z_s^{-1},\dots,z_k^{-1}\} \cup \bigcup_{1\leq i\leq k} \left(\pref(p_i)\cup \pref(q_i^{-1})\right).
$$
Next observe that for any $1\leq j\leq r$ we can write $u \equiv u(z_1,\dots,z_k)\equiv u'z_ju''$, which means that  $
z_j^{-1} = u''u'
$ holds in $H$, since $u(z_1, \ldots, z_k)=1$ in $H$. But both $u',u''$ are products of letters from $\{z_1,\dots,z_r,z_s^{-1},\dots,z_k^{-1}\}$, which shows that $z_j^{-1}\in P_w$. Similarly, $z_j\in P_w$ for all $s\leq j\leq k$. Since $r \leq s+1$ this proves that $\ol{Z} \subseteq P_w$ and hence the entire group $H$ is contained in $P_w$. Thus $P_w$ is equal to the submonoid of $FG(X_1)*H$ generated by the set $
H 
\cup 
Q
$ where 
$$
Q=\bigcup_{1\leq i\leq k} \left(\pref(p_i)\cup \pref(q_i^{-1})\right) \subseteq FG(X_1).
$$ 
To complete the proof it will suffice to show that the conditions of Corollary~\ref{cor:free-pr} are satisfied for the submonoid $P_w$ of the group $G = FG(X_1) * H$. The groups $FG(X_1)$ and $H$ both have decidable word problem by Magnus' Theorem. 

We have  $P_w \cap H = H$ which is finitely generated. We claim that $P_w \cap FG(X_1)$ equal to the submonoid of $FG(X_1)$ generated by $Q$, and hence is finite generated. 

Indeed, we have that $P_w$ is equal to the submonoid of $FG(X_1)*H$ generated by the set $
H 
\cup 
Q
$. Let $g \in P_w \cap FG(X_1)$. Since $g \in P_w$ we can write $g = h_0 t_1 h_1 \dots  h_l t_l$ where  $h_i \in H$ and $t_i \in \Mon\gen{Q}$, and, furthermore, $h_i \neq 1$ for $1 \leq i \leq l-1$ and $t_i \neq 1$ for all $i$. In the free product $FG(X_1) \ast H$ this is a reduced form. Since we are assuming that $g \in FG(X_1)$ it follows by Lemma~\ref{lem:rf-amal} that we must have $g=t_1 \in \Mon\gen {Q}$. This proves that $P_w \cap FG(X_1)$ is contained in $\Mon\gen Q$. The opposite containment is trivial. Hence condition (i) holds. 

Condition (iii) holds again since $P_w \cap H = H$, while condition (ii) holds by Benois' Theorem as $P_w \cap FG(X_1)$ is a finitely generated submonoid of the free group $FG(X_1)$. This completes the proof of the theorem.
\end{proof}

\begin{exa}\label{exa:marker}
Let $X = \{a,b,x,y\}$, let $w = axbaybaybaxbaybaxb$ and set $G = \Gp\pre{X}{w=1}$ and $M = \Inv\pre{X}{w=1}$. 
Since $axb$ is both a prefix and a suffix of $w$ it follows that this word represents an invertible element $M$. It follows that the word  $(ayb)aybaxb(ayb)$ also represents an invertible element of $M$ and hence so does the word $ayb$.  
We conclude that 
$$
w=(axb)(ayb)(ayb)(axb)(ayb)(axb)
$$
is a unital factorisation and thus also a conservative factorisation by Theorem~\ref{thm:UnitalConservative}. 
Notice that $x$ occurs exactly once in $axb$ but not in $ayb$, and conversely, $y$ occurs just once in $ayb$ but not in $axb$. 
So, the above factorisation of $w$ satisfies the unique marker letter condition of 
Theorem~\ref{thm:marker}.
Also note that $w$ is a cyclically reduced word. 
Therefore applying the theorem we conclude that the group defined by the presentation 
\[
\Gp\pre{a,b,x,y}{axbaybaybaxbaybaxb=1}
\]
has decidable prefix membership problem and the inverse monoid 
\[
\Inv\pre{a,b,x,y}{axbaybaybaxbaybaxb=1}
\]
has decidable word problem. 
\end{exa}

Many other similar examples to which Theorem~\ref{thm:marker} can be applied may be constructed. 
In the example above in order to deduce that the inverse monoid has decidable word problem we just used the fact that it is $E$-unitary since the defining relator is a cyclically reduced word. It was not important that the defining relator was a positive word. 
For example, 
in much the same way we can show that the
inverse monoid
$$
M' = \Inv\pre{a,b,x,y}{a^{-1}xbab^{-1}ayb^{-1}b^{-1}ayb^{-1}a^{-1}b^{-1}x^{-1}ab^{-1}ayb^{-1}a^{-1}xba=1}.
$$
has decidable word problem.

Next we shall see that Theorem~\ref{thm:marker} can also be applied in certain situations where the given defining relator does not immediately satisfy the unique marker letter condition. 

\begin{exa}\label{exa:OHareDecWP}
Let $M$ be the  ``O'Hare inverse monoid'' 
$$
\Inv\pre{a,b,c,d}{abcdacdadabbcdacd=1}.
$$
and let $G$ be the group with the same presentation. 
Recall in Example \ref{exa:ohare} where we defined the ``O'Hare inverse monoid'' 
we saw that 
$$
w\equiv (abcd)(acd)(ad)(abbcd)(acd)
$$
is a unital factorisation and thus also a conservative one. 
Note that these invertible pieces do not satisfy the unique marker letter property. 
However, as we shall now see, this monoid admits a one relator special inverse monoid presentation such that the defining relator does satisfy the hypotheses of Theorem~\ref{thm:marker}. 
In fact, we shall identify an infinite family of examples, which includes the O'Hare monoid, for which this approach is possible.

\begin{pro}\label{pro:OHare}
Let 
$$
M = \Inv\pre{X}{au_{i_1}dau_{i_2}d\dots au_{i_m}d = 1},
$$
where $a,d\in X$ and $u_{i_k} \in \ol{Y}^*$ is a reduced word for $1 \leq k \leq m$
where $Y = X \setminus \{a,d\}$. 

Assume further that the 
following conditions hold:
\begin{itemize}
\item[(i)] for some $1\leq j\leq m$, $u_{i_j}$ is the empty word;
\item[(ii)] for each $x\in X\setminus\{a,d\}$ there exist $r,s$ such that $x\equiv\red(u_{i_r}u_{i_s}^{-1})$;
\item[(iii)] each word $au_{i_k}d$ represents an invertible element of $M$.
\end{itemize}
Then the group defined by the presentation
$$
G = \Gp\pre{X}{au_{i_1}dau_{i_2}d\dots au_{i_m}d = 1},
$$
has decidable prefix membership problem, and 
the inverse monoid 
$M$ has decidable word problem.
\end{pro}

\begin{proof}
By (iii) all of the words $a u_{i_k} d$ with $1 \leq k \leq m$ all represent invertible elements of $M$. 
It follows that the inverse words $(a u_{i_k} d)^{-1}\equiv d^{-1}u_{i_k}^{-1}a^{-1}$ with $1 \leq k \leq m$ also all represent invertible elements of $M$. 
Since the product of two invertible elements is invertible, it follows that for all $1 \leq r,s \leq m$ the word $(a u_{i_r} d)(d^{-1} u_{i_s}^{-1} a^{-1})$ represents an invertible element of $M$. 

We then use the following well-known fact from inverse semigroup theory: if $bc$ is a right invertible element of an inverse monoid,
then $bcc^{-1}=b$. 
Since every prefix of the word $(a u_{i_r} d)(d^{-1} u_{i_s}^{-1} a^{-1})$ is right invertible, applying the above general fact we conclude that 
\[
(a u_{i_r} d)(d^{-1} u_{i_s}^{-1} a^{-1}) = a \red(u_{i_r} u_{i_s}^{-1}) a^{-1} 
\]
holds in $M$ for all $1 \leq r,s \leq m$.  

By conditions (ii) it follows that
for every letter $x \in Y$ we have that $axa^{-1}$ represents an invertible element of $M$. Also, by condition (i) and (iii) the word $ad$ represents an invertible element of $M$. 

For each $1 \leq r \leq m$ write 
$$
au_{i_r}d\equiv ab_{1,r}\dots b_{t_r,r}d
$$
where $b_{i,r} \in Y$ for $1 \leq i \leq t_r$. 
Using the observations from the previous paragraph, and the general observation above about cancelling inverse pairs in right invertible words we conclude that in $M$ we have 
\begin{equation}\label{eqn:OHare}
au_{i_r}d = ab_{1,r}\dots b_{t_r,r}d = (ab_{1,r}a^{-1})\dots(ab_{t_r,r}a^{-1})(ad).
\end{equation}
for all $1 \leq r \leq m$. 

Let 
$v_r \equiv (ab_{1,r}a^{-1})\dots(ab_{t_r,r}a^{-1})(ad)$ for all $1 \leq r \leq m$ 
and then set $w' \equiv v_1 v_2 \dots v_k$. 
We claim that the presentations $\Inv\pre{X}{w=1}$ and $\Inv\pre{X}{w'=1}$ are equivalent in the sense that the identity map on $\overline{X}$ induces an isomorphism between the inverse monoids defined by these presentations.  
To prove this it suffices to show that $w'=1$ holds in the monoid $M=\Inv\pre{X}{w=1}$, and, conversely, that $w=1$ holds in the monoids $M'=\Inv\pre{X}{w'=1}$. 

The fact that $w'=1$ holds in $M$ follows immediately from Equation~\ref{eqn:OHare}.  
Conversely, in the inverse monoid $M'=\Inv\pre{X}{w'=1}$
each prefix of $w'$ arising as a product of factors of the form $ab_{j,r}a^{-1}$ represents a right invertible element of $M'$.
This observation along with the 
general fact above about cancelling inverse pairs in right invertible words 
makes it possible to delete from $w'$ all factors of the form $a^{-1}a$ without
changing the value of $w'$ in $M'$. In other words, $w=1$ holds in $M'$. 

This shows that the presentations for $M$ and $M'$ are equivalent, and in particular that $M$ and $M'$ are isomorphic via the identity map on $\ol{X}$.  
It follows that for any word $\gamma \in \ol{X}^*$ we have that $\gamma$ represents a right invertible element of $M$ if and only if $\gamma$ represents a right invertible element of $M'$.
Let $R$ be the submonoid of right invertible elements of $M$, and let $R'$ be the submonoid of right invertible elements of $M'$.  
Let $\phi: M \rightarrow G$ and $\phi' : M' \rightarrow G$  be the maps to the maximal group image induced by the identity map on $\ol{X}$. 
Then we have 
\[
P_w = R \phi = R' \phi' = P_{w'}. 
\]

However, the relator word $w'$ from the presentation of $M'$ has a unital and thus conservative
factorisation into factors of the form $axa^{-1}$, $x\in X\setminus\{a,d\}$, and $ad$. Picking $x$ as the unique marker letter
from $axa^{-1}$, and $d$ from $ad$, shows that the inverse monoid $M'$ has a presentation which satisfies the conditions of
Theorem \ref{thm:marker}. Hence, the 
membership problem for 
$P_{w'} = P_w$ in $G$ is decidable. 
Hence the group presentation $\Gp\pre{X}{w=1}$ has decidable prefix membership problem and the inverse monoid $M$ has decidable word problem.  
\end{proof}

In particular the above proposition applies the O'Hare monoid 
\[
\Inv\pre{a,b,c,d}{(abcd)(acd)(ad)(abbcd)(acd)=1}
\]
since, as already observed, the displayed decomposition is unital, and hence in particular $ad$ represents an invertible element of the monoid, and, moreover, we clearly have 
$b=\red((bc)c^{-1})$ and $c=\red(c1^{-1})$. Hence all the hypotheses of the proposition are satisfied and we conclude that the O'Hare monoid has decidable word problem. 
\end{exa}
 
\begin{rmk}
It was pointed out to us by Jim Howie (Heriot-Watt University, Edinburgh) \cite{Jim} that the ``O'Hare group'' 
$\Gp\pre{a,b,c,d}{abcdacdadabbcdacd=1}$ is in fact a free group of rank 3 (albeit in a rather non-obvious way). 
Therefore, in this case $P_w$ has decidable membership in this group as a direct consequence of Benois' Theorem, and so the 
word problem for the O'Hare inverse monoid is decidable. 
On the other hand, not every group satisfying the hypotheses of Proposition~\ref{pro:OHare} is free.
\end{rmk}

\subsection{Disjoint alphabets theorem}

\begin{lem}\label{lem:useful}
Let $G=B *_A C$ and let $U$ be a finite subset of $B\cup C$ such that $M=\Mon\gen{U}$, the submonoid of $G$ generated by $U$,
contains $A$. Then $M\cap B$ is generated by $(U\cap B)\cup A$ and $M\cap C$ is generated by $(U\cap C)\cup A$. Consequently,
if $A$ is finitely generated, then so are the monoids $M\cap B$ and $M\cap C$.
\end{lem}

\begin{proof}
Let $g\in M\cap B$. Then, since $g\in M$, we may write
$$
g = c_0 b_1c_1 \dots b_lc_l
$$
for some $b_i\in\Mon\gen{U\cap B}$ and $c_i\in\Mon\gen{U\cap C}$ such that $c_i\neq 1$ for $1\leq i\leq l-1$ and $b_i\neq 1$ for all $i$.
This expression is either a reduced form in $G$, whence $l=1$, $c_0=c_1=1$ and $g=b_1\in\Mon\gen{U\cap B}\subseteq \Mon\gen{(U\cap B)\cup A}$ 
by Lemma 4.2, or, otherwise, at least one of the terms belongs to $A$. For example, suppose that $b_j\in A$. Then $c_{j-1}' = c_{j-1}b_jc_j\in 
\Mon\gen{(M\cap C)\cup A}$, so upon relabelling $b_k'=b_k$ for all $1\leq k\leq j-1$, $c_k'=c_k$ for all $0\leq k\leq j-2$, $b_k'=b_{k+1}$
for all $j+1\leq k\leq l-1$, and $c_k'=c_{k+1}$ for all $j\leq k\leq l-1$, we get a new expression
$$
g = c'_0 b'_1c'_1 \dots b'_{l-1}c'_{l-1},
$$
where $b_k'\in\Mon\gen{(U\cap B)\cup A}$ and $c_k'\in\Mon\gen{(U\cap C)\cup A}$ for all $k$, for which the previous argument can be repeated. 
We proceed similarly if $c_j\in A$ for some $j$. Continuing in this fashion, we eventually arrive at a reduced form for $g$, leading to the 
conclusion that $g\in \Mon\gen{(U\cap B)\cup A}$, as required.

Analogously, if $g\in M\cap C$ we have that $g\in \Mon\gen{(U\cap C)\cup A}$, thus proving the lemma.
\end{proof}

Here is our second application, whose proof makes an appeal to Theorem \ref{thm:amal}.

\begin{thm}\label{thm:disj} 
Let $G=\Gp\pre{X}{w=1}$ 
where $w \in \ol{X}^*$ is a cyclically reduced word. 
Suppose that there is a finite alphabet 
$Y=\{y_1,\dots,y_k\}$ with $k \geq 2$ 
and a word 
$u\in\ol{Y}^*$
such that
$w\equiv u(w_1,\dots,w_k)$,
all letters from $Y$ appear in $u$ (either in positive or inverted form), 
and that this determines a conservative factorisation of $w$. 
Suppose that for any pair of distinct $i, j \in \{ 1, \dots, k \}$ the words $w_i$ and $w_j$ have no letters in common.
Then the group $G= \Gp\pre{X}{w=1}$ has decidable prefix membership problem and thus 
$\Inv\pre{X}{w=1}$ has decidable word problem. 
\end{thm}

\begin{proof}
For $1\leq i\leq k$, let $X_i\subseteq X$ denote the \emph{content} of $w_i$, namely the set of all letters
$x\in X$ that appear in $w_i$. 
Then the conditions given in the theorem state
that $i\neq j$ implies $X_i\cap X_j=\es$. We also let $X_0=X\setminus\bigcup_{1\leq i\leq k}X_i$. 
Notice that
since $w$ is cyclically reduced it follows that $u$ must also be cyclically reduced.

Let $t_1\not\in X$ be a new letter. Then an easy application of Tietze transformations gives
$$
G = \Gp\pre{X,t_1}{w_1t_1^{-1}=1,\ u(t_1,w_2,\dots,w_k)=1}.
$$
Let 
$G_1=\Gp\pre{X\setminus X_1,t_1}{u(t_1,w_2,\dots,w_k)=1}$ 
noting that by the assumptions the word 
$u(t_1,w_2,\dots,w_k)$ is written over the alphabet  
$(X\setminus X_1) \cup \{t_1\}$ 
and is a cyclically reduced word. 
Since by assumption $k \geq 2$, it follows by Magnus' Freiheitssatz 
that the subgroup $A_1'=\Gp\gen{t_1}$ of $G_1$ generated by $t_1$ is an infinite cyclic group. 
On the other hand, since $w_1\in\ol{X_1}^*$ is a non-empty reduced word, it follows that the subgroup 
$A_1 = \Gp\gen{w_1}$ of free group $FG(X_1)$ generated by $w_1$ is 
also infinite cyclic. 
Thus we can form the amalgamated free product 
\[
FG(X_1) *_{A_1} G_1.
\]
there $A_1$ and $A_1'$ are identified via the isomorphism sending $w_1$ to $t_1$.
This gives 
\[
G = 
\Gp\pre{
X_1, (X \setminus X_1), t_1}{
w_1 = t_1, \ u(t_1,w_2,\dots,w_k)=1} =
FG(X_1) *_{A_1} G_1.
\]

The same reasoning as above can be then applied to $G_1$: upon introducing a new letter $t_2$, one can decompose
$G_1$ as the free product of $FG(X_2)$ and $G_2=\Gp\pre{X\setminus(X_1\cup X_2),t_1,t_2}{u(t_1,t_2,w_3,\dots,w_k)=1}$
amalgamated over the infinite cyclic subgroups $A_2$ and $A_2'$ generated by $w_2$ and $t_2$, respectively. 

Continuing in this way, 
after a finite number of steps 
we obtain the following tower of amalgamated free products:
$$
G = FG(X_1) *_{A_1} (FG(X_2) *_{A_2} (\dots (FG(X_k)*_{A_k} G_k)\dots)),
$$
where $A_i$ is generated by $t_i=w_i$ and
\[
G_k=\Gp\pre{X_0,t_1,\dots,t_k}{u(t_1,\dots,t_k)=1}
= FG(X_0) * H   
\]
where 
$$
H=\Gp\pre{t_1,\dots,t_k}{u(t_1,\dots,t_k)=1}.
$$

For each $0 \leq i \leq k$ we shall now define a submonoid $M_i$ of $G_i$ inductively. 
The sequence of monoids we define will have the property that for all $i$ we have 
\[
M_{i+1} 
= 
M_i \cap G_{i+1}, 
\]
and also that $M_0 = P_w$. 
We will prove using Theorem~\ref{thm:amal} that the membership problem for $M_i$ in $G_i$ is decidable for all $i$, and since $M_0 = P_w$ this will suffice to complete the proof of the theorem. 

For all $1 \leq i \leq k$ set 
$$
W_{i} = \left\{\begin{array}{ll}
\pref(w_{i}) & \text{if }u(t_1,\dots,t_k)\text{ contains }t_{i}\text{ but not }t_{i}^{-1},\\
\pref(w_{i}^{-1}) & \text{if }u(t_1,\dots,t_k)\text{ contains }t_{i}^{-1}\text{ but not }t_{i},\\
\pref(w_{i})\cup \pref(w_{i}^{-1}) & \text{if }u(t_1,\dots,t_k)\text{ contains both }t_{i},t_{i}^{-1}.
\end{array}\right.
$$
Then set 
\begin{align*}
M_k  =\Mon\langle t_i^\eps:\  & 1\leq i\leq k\text{ and }\eps\in\{1,-1\}\\ 
& \text{such that }t_i^\eps\text{ is contained in }u(t_1,\dots,t_k)\rangle
\end{align*}
which is a submonoid of $G_k = FG(X_0) * H$, 
and define inductively 
\[
M_{i-1} = \Mon\gen{M_{i}\cup W_{i}} \leq G_{i-1} =
FG(X_i)*_{A_i} G_i
\]
for $1 \leq i \leq k$.

It may be shown that in fact $M_k = H$ 
using a similar argument as in the proof of Theorem \ref{thm:marker}.
Indeed, if, for example,
$t_i$ occurs in $u(t_1,\dots,t_k)$ but not $t_i^{-1}$, then one can write $u(t_1,\dots,t_k)\equiv u't_iu''$. Therefore, in 
$G_k$ we have $t_i^{-1}=u''u'\in M_k$. This shows that $t_i,t_i^{-1}\in M_k$ for all $1\leq i\leq k$, so the required conclusion $M_k=H$ follows.

Since the word $u$ defines a conservative factorisation of $w$,
it follows that the prefix monoid $P_w$
is equal to the submonoid of
$G = G_0$ generated by 
$\cup_{1 \leq i \leq k} W_i$. Now by definition we have 
\[
M_0 =  
\mathrm{Mon} \langle 
\mathrm{Mon} \langle 
\dots
\mathrm{Mon} \langle 
\mathrm{Mon} \langle 
H \cup W_k
\rangle 
\cup 
W_{k-1} 
\rangle 
\dots 
\rangle 
\cup 
W_1 
\rangle.
\]
From this, using the natural embeddings of $G_{i-1}$ into $G_i$ for all $i$ 
it may be verified that in $G= G_0$ we have  
\[
M_0 = \Mon\gen{W_1 \cup W_2 \cup \dots \cup W_k} = P_w. 
\]

So, it remains to argue by
induction that 
the memberhsip problem for 
$M_i$ 
withing $G_i$ is 
decidable 
for all $i$. 
Clearly, each $M_i$ is finitely generated, say $M_i=\Mon\gen{U_i}$ for some finite subset $U_i\subset M_i$.
Since both $FG(X_0)$ and the one-relator group $H$ have decidable word problems, the latter by Magnus' Theorem, and  $M_k\cap FG(X_0)=
\{1\}$ and $M_k\cap H=H$, we can apply Corollary \ref{cor:free-pr} 
to deduce that 
the membership 
problem for $M_k$ in $G_k$ is decidable.

Now assume inductively that the membership problem for $M_{i+1}$ 
in $G_{i+1}$ 
is decidable 
for some $i<k$. The latter is a one-relator group,
so it has decidable word problem, as does $FG(X_{i+1})$. 
Furthermore, since 
$FG(X_{i+1})$ 
is a free group, it follows from Benois' Theorem that the membership problem for 
$A_{i+1}=\Gp\gen{w_{i+1}}$ 
in  
$FG(X_{i+1})$ 
is decidable.
Since $t_{i+1}$ is one of its generators,  
it follows by   \cite[Theorem IV.5.3]{LSch}
that the membership problem for 
$A_{i+1}'=\Gp\gen{t_{i+1}}$ 
in the one-relator group $G_{i+1}$ is decidable. 

We claim that 
$M_i\cap G_{i+1}=M_{i+1}$.
Indeed, we have $G_i = FG(X_{i+1}) *_{A_{i+1}} G_{i+1}$, while $M_i=\Mon\gen{M_{i+1}\cup W_{i+1}}=
\Mon\gen{U_{i+1}\cup W_{i+1}}$. Therefore, by Lemma~\ref{lem:useful}, $M_i\cap G_{i+1}$ is generated by
$((U_{i+1}\cup W_{i+1})\cap G_{i+1})\cup A_{i+1}' = U_{i+1}\cup A_{i+1}'$ and thus $M_i\cap G_{i+1} = 
\Mon\gen{U_{i+1}\cup \{t_{i+1},t_{i+1}^{-1}\}} = M_{i+1}$ because $t_{i+1},t_{i+1}^{-1}\in H \subseteq M_{i+1}$. 
Similarly, 
Lemma~\ref{lem:useful}
yields that $M_i\cap FG(X_{i+1})$ is finitely generated by $W_{i+1}$.
This means that $M_i$ is indeed generated by $(M_i\cap G_{i+1})\cup (M_i\cap FG(X_{i+1}))$, and, furthermore, 
the membership problem for 
$M_{i+1} = M_i\cap G_{i+1}$ 
in $G_{i+1}$ 
is decidable 
by the inductive hypothesis,
while 
the membership problem for 
$M_i\cap FG(X_{i+1})$ in $FG(X_{i+1})$ 
is decidable by
by Benois' Theorem. 

Hence the hypotheses of Theorem \ref{thm:amal}
are satisfied, and applying this theorem we 
conclude that 
the membership problem for 
$M_i$ in $G_i$ is 
decidable. In particular, the membership problem for $M_0=P_w$ in $G_0=G$ is decidable, 
so the theorem follows.
\end{proof}

\begin{exa}
Consider the inverse monoid
$$
M = \Inv\pre{a,b,c,d}{ababcdcdababcdcdcdcdabab=1}.
$$
Then, if we denote $\alpha\equiv abab$ and $\beta\equiv cdcd$, the relator word becomes $w\equiv\alpha\beta\alpha\beta\beta\alpha$,
and we conclude, just as in Example \ref{exa:marker}, that both $\alpha$ and $\beta$ represent invertible elements of $M$.
Hence,
$$
w\equiv (abab)(cdcd)(abab)(cdcd)(cdcd)(abab)
$$
is a conservative factorisation that satisfies the conditions of Theorem \ref{thm:disj}. It follows that $G=\Gp\pre{a,b,c,d}
{\alpha\beta\alpha\beta\beta\alpha=1}$ has decidable prefix membership problem and that the word problem of $M$ is decidable, too.
In more detail,
$$
G = FG(a,b)*_{A_1}\left(FG(c,d)*_{A_2}\Gp\pre{t,s}{tstsst=1}\right),
$$
where $A_1$ and $A_2$ are infinite cyclic groups generated by $t=abab$ and $s=cdcd$, respectively, and the prefix monoid is 
generated by $t,t^{-1}=stsst,s,s^{-1}=tsstt$ (thus containing the whole group $\Gp\pre{t,s}{tstsst=1}$) and $a,ab,aba,c,cd,cdc$ 
(here $aba$ and $cdc$ are obviously redundant).
\end{exa}

\begin{exa}
We finish the subsection by a non-example, showing the significance of the disjoint content condition in Theorem \ref{thm:disj}.
Namely, let
$$
M = \Inv\pre{a,b,c}{ababbcbcbbababbcbcbbcbcbbababb=1}.
$$
Just as in previous examples, it is easy to see that 
$$
w = (ababb)(cbcbb)(ababb)(cbcbb)(cbcbb)(ababb)
$$
is a unital factorisation of the relator word, but the pieces do not have disjoint content (they have the letter $b$ in common).
As $w$ is cyclically reduced, the word problem for $M$ reduces to the prefix membership problem for the group $G$ defined by the presentation
$\Gp\pre{a,b,c}{\alpha\beta\alpha\beta\beta\alpha=1}$, where $\alpha\equiv ababb$ and $\beta\equiv cbcbb$. We can now replace
its sole relation by $ababbs^{-1}=1$ and $scbcbbscbcbbcbcbbs=1$ to obtain the free amalgamated decomposition
$$
G = FG(a,b)*_A \Gp\pre{c,d,s}{scbcbbscbcbbcbcbbs=1},
$$
where $A$ is a joint free subgroup of rank 2 generated by $b$ and $ababb=s$. (Note that the second factor satisfies the requirements
for applying Theorem \ref{thm:disj}.)

However, the prefix submonoid $P_w$ of $G$ is generated (after removing some obvious redundancies) by $s,s^{-1},a,ab,c,cb,cbcbb$.
It also contains $(cbcbb)^{-1}=(scbcbb)^2$ and thus $b^{-1}=(cbcbb)^{-1}(cb)^2$. The problem is that we don't know if $A\subseteq P_w$:
for this we would need $b\in P_w$ (which seems likely not to hold). Therefore, at present it seems that Theorem \ref{thm:amal} cannot
be applied to this case. 
\end{exa}

\subsection{Cyclically pinched presentations}

Following e.g.\ \cite{CFR,FRS} we say that a one-relator group $G$ is \emph{cyclically pinched} if it is defined by a presentation of the
form 
$$
\pre{X\cup Y}{u=v}
$$
where $u\in \ol{X}^*$ and $v\in\ol{Y}^*$ are nonempty reduced words, with the alphabets $X$ and $Y$ disjoint. Clearly, the defining relation
is equivalent to $uv^{-1}=1$, and in this sense we are going to refer to $\Gp\pre{X\cup Y}{uv^{-1}=1}$ as a \emph{cyclically pinched presentation}.

\begin{thm}\label{thm:pinch1}
The prefix membership problem is decidable for any group defined by a cyclically pinched presentation
$$
\Gp\pre{X\cup Y}{uv^{-1}=1}.
$$
Consequently, the word problem is decidable for all one-relator inverse monoids of the form 
$$
\Inv\pre{X\cup Y}{uv^{-1}=1}
$$
with $u\in\ol{X}^*$ and $v\in\ol{Y}^*$ both reduced words.
\end{thm}

\begin{proof}
As is well-known, a cyclically pinched group $G$ is a free product of free groups $FG(X)$ and $FG(Y)$ amalgamated over an infinite cyclic group $A$ such 
that $Af$ is generated by $u$ and $Ag$ is generated by $v$. Hence, it suffices to check if the conditions of Theorem A are satisfied for the prefix monoid
$P_w$ where $w\equiv uv^{-1}$. Indeed, the latter monoid is generated in $G$ by
$$
\pref(u) \cup u\cdot\pref(v^{-1}).
$$
Note that the set $u\cdot\pref(v^{-1})$ is in $G$ actually equal to $\pref(v)$. Since the generating set of $P_w$ contains $u$ (and thus $v$), the monoid
$P_w$ contains the whole amalgamated subgroup. Since this subgroup is finitely generated, Lemma \ref{lem:useful} implies that so are the monoids $P_w\cap FG(X)=
\Mon\gen{\pref(u)}$ and $P_w\cap FG(Y)=\Mon\gen{\pref(v)}$. Hence, the conditions (i) and (ii) of Theorem A are satisfied, and the remaining conditions hold 
by Benois' Theorem. Therefore, the membership problem for $P_w$ in $M$ is decidable.
\end{proof}

\begin{exa}\label{exa:surface}
Both the orientable surface group 
$$
\Gp\pre{a_1,\dots,a_n,b_1,\dots,b_n}{[a_1,b_1]\dots[a_n,b_n]=1}
$$
and the non-orientable surface group
$$
\Gp\pre{a_1,\dots,a_n}{a_1^2\dots a_n^2=1}
$$
of genus $n\geq 2$ are cyclically pinched (e.g.\ for $u\equiv [a_1,b_1]\dots[a_{n-1},b_{n-1}]$ and $v\equiv [a_n,b_n]^{-1}$ in the orientable
case, and for $u\equiv a_1^2\dots a_{n-1}^2$ and $v\equiv a_n^{-2}$ for the non-orientable case). 
Hence, the corresponding prefix membership problems 
are decidable.
For the first family of presentations above there are already several proofs in the literature that the prefix membership problem is decidable; see \cite{IMM,MMSu,Mea}.
\end{exa}

\section{Deciding membership in submonoids of HNN extensions}
\label{sec:hnn}

Our aim in this section is to give two general decidability results concerning the membership problem for certain submonoids of HNN extensions of finitely generated groups. Then in the next section these results will be applied to the prefix membership problem for certain one-relator groups.

Before stating the main results we first recall some definitions and fix some notation which will remain in place for the rest of the section.  Let $G$ be a group finitely generated by $X$ with canonical homomorphism $\pi: \ol{X}^* \rightarrow G$.  Let $A$ and $B$ be two isomorphic finitely generated subgroups of $G$ and let $\phi: A \rightarrow B$ be an isomorphism. Moreover, let $Y = \{y_1, \ldots, y_k \}$ and $Z = \{z_1, \ldots, z_k\}$ be, respectively, finite generating sets for $A$ and $B$, with canonical homomorphisms $\theta: \ol{Y}^* \rightarrow A$, $\xi: \ol{Z}^* \rightarrow B$ such that $y_i \theta \phi = z_i \xi$ for $1 \leq i \leq k$.   Set $a_i = y_i \theta$ and $b_i = z_i \xi$ for $1 \leq i \leq k$. Observe that $\{a_1, \ldots, a_k\}$ is a finite subset of $A$ which generates $A$, and similarly $\{b_1, \ldots, b_k\}$ is a finite subset of $B$ generating $B$. 

We use 
$$
G^* = G *_{t,\phi:A\to B}
$$
to denote the HNN extension of $G$ 
with a stable letter $t\not\in X$ and associated finitely generated subgroups $A,B$. 
So $G^*$ is the group obtained by taking a presentation for $G$ with respect to $X$, adding $t$ as a new generator, and adding the relations 
\[
t^{-1}a_it = b_i
\]
for $1\leq i\leq k$. Formally these relations should be written over the alphabet $\ol{X \cup \{t\}}$, that is, as $t^{-1} u_i t = v_i$ where $u_i, v_i \in \ol{X}^*$ satisfy $u_i \pi = a_i$ and $v_i \pi = b_i$. Throughout this section $u_i$ and $v_i$ will denote words with these properties.  Note that for any word $w(u_1, \ldots, u_k) \in \ol{X}^*$ representing the element $a\in A$, the word $w(v_1,\ldots,v_k)\in \ol{X}^*$ represents the element $a\phi\in B$.

We now state the two main results of this section.

\begin{theoremletter}\label{thm:hnn41}
Let $G^*=G *_{t,\phi:A\to B}$ be an HNN extension of a finitely generated group $G$ such that $A,B$ are also finitely
generated. Assume that $G$ has decidable word problem and that the membership problems of $A$ and $B$ in $G$ are decidable.
Let $M$ be a submonoid of $G^*$ such that the following conditions hold:
\begin{itemize}
\item[(i)] $A\cup B\subseteq M$;
\item[(ii)] $M\cap G$ is finitely generated, and
            $$M=\Mon\gen{(M\cap G)\cup\{t,t^{-1}\}};$$
\item[(iii)] the membership problem for $M\cap G$ in $G$ is decidable.
\end{itemize}
Then the membership problem for $M$ in $G^*$ is decidable.
\end{theoremletter}

\begin{theoremletter}\label{thm:hnn}
Let $G^*=G *_{t,\phi:A\to B}$ be an HNN extension of a finitely generated group $G$ such that $A,B$ are also finitely
generated. Assume that the following conditions hold:
\begin{itemize}
\item[(i)] the rational subset membership problem is decidable in $G$;
\item[(ii)] $A \leq G$ is \EC.
\end{itemize}
Then for any finite $W_0,W_1,\dots,W_d,W_1',\dots,W_d'\subseteq G$, $d\geq 0$, the membership problem for 
$$
M=\Mon\gen{W_0\cup W_1t\cup W_2t^2\cup\dots\cup W_dt^d\cup tW_1'\cup\dots\cup t^dW_d'}
$$ 
in $G^*$ is decidable.
\end{theoremletter}

\begin{rmk}\label{rmk:inv}
The conclusion of the previous theorem also holds if we replace $t$ by $t^{-1}$ in the generating set of the monoid $M$.
Namely, it is straightforward to see that 
$$
\Mon\gen{W_0\cup W_1t^{-1}\cup W_2t^{-2}\cup\dots\cup W_dt^{-d}\cup t^{-1}W_1'\cup\dots\cup t^{-d}W_d'}
$$
is in fact equal to 
$$
\left(\Mon\gen{W_0^{-1}\cup (W_1')^{-1}t\cup (W_2')^{-1}t^2\cup\dots\cup (W_d')^{-1}t^d\cup tW_1^{-1}\cup\dots\cup t^dW_d^{-1}}\right)^{-1},
$$
which enables us to invoke Theorem \ref{thm:hnn}.
\end{rmk}

By standard results on HNN extensions (see e.g. \cite[Ch.\ IV]{LSch}) every element  $g\in G*_{t,\phi:A\to B}$ can be written written as 
$$
g = g_0t^{\eps_1}g_1t^{\eps_2}\dots t^{\eps_n}g_n,
$$
where $g_0,g_1,\dots,g_n\in G$ and $\eps_i\in\{1,-1\}$ for all $1\leq i\leq n$. This expression is said to be \emph{reduced} if for all $1\leq i\leq n-1$ we have $g_i\not\in A$ whenever $\eps_i=-1$ and $\eps_{i+1}=1$, and $g_i\not\in B$ whenever $\eps_i=1$ and $\eps_{i+1}=-1$. 
Similarly, for words $w_1, w_2, \ldots, w_n \in \ol{X}^*$ and $\eps_i \in \{1, -1\}$ $(1 \leq i \leq n)$ the word 
\[
w_0t^{\eps_1}w_1t^{\eps_2}\dots t^{\eps_n}w_n,
\]
is said to be \emph{a word in reduced form} if and only if 
$(w_0 \pi) t^{\eps_1}(w_1 \pi)t^{\eps_2}\dots t^{\eps_n}(w_n \pi)$
is a reduced expression.
The following result is standard and can be proved e.g. by applying \cite[Lemma IV.2.3]{LSch} and Britton's Lemma. 

\begin{lem}\label{lem:Britton}
An equality of two reduced forms
$$
g_0t^{\eps_1}g_1t^{\eps_2}\dots t^{\eps_n}g_n = h_0t^{\delta_1}h_1t^{\delta_2}\dots t^{\delta_m}h_m
$$
holds in the HNN extension $G*_{t,\phi:A\to B}$ if and only if $n=m$, $\eps_i=\delta_i$ for all $1\leq i\leq n$, and there exist $1=\alpha_0,\alpha_1,\dots,\alpha_n,\alpha_{n+1}=1\in A\cup B$ such that for all $0\leq i\leq n$ we have $\alpha_i\in A$ if $\eps_i=-1$, $\alpha_i\in B$ if $\eps_i=1$, and
$$
h_i = \alpha_i^{-1}g_i(t^{\eps_{i+1}}\alpha_{i+1}t^{-\eps_{i+1}}).
$$
\end{lem}

\begin{lem}\label{lem:HNN:reduced}
Let $G^*=G *_{t,\phi:A\to B}$ and let $M$ be a submonoid of $G^*$ such that 
\[
M = \Mon{(M \cap G) \cup \{t, t^{-1} \}}.  
\]
Then every element $g \in M$ can be written in reduced form  
$$
g = g_0t^{\eps_1}g_1t^{\eps_2}\dots t^{\eps_n}g_n
$$
such that $g_i\in M\cap G$  for all $0 \leq i \leq n$. 
\end{lem}
\begin{proof}
Assume that $g\in M$. Then by assumption  
$$
g = h_0t^{\delta_1}h_1t^{\delta_2}\dots t^{\delta_k}h_k
$$
for some $h_i\in M\cap G$, $0\leq i\leq k$, and $\delta_i\in\{1,-1\}$, $1\leq i\leq k$. By \cite[page 184]{LSch}, the above product, which itself is not necessarily reduced, can be transformed into a reduced form by applying a finite number of $t$-reductions. Recall that $t$-reductions the following operations 
\begin{itemize}
    \item replace $t^{-1} g t$, where $g \in A$, by $g \phi$, or
    \item replace $t g t^{-1}$, where $g \in B$, by $g \phi^{-1}$. 
\end{itemize}
Hence to prove the lemma it suffices to show that, under our assumptions, applications of $t$-reductions to a product of the above form preserves the property of the $G$-terms belonging to $M$. 

Consider a product 
$p_0t^{\gamma_1}p_1t^{\gamma_2}\dots t^{\gamma_m}p_m$, 
such that $p_i \in M$ for $1 \leq i \leq m$, to which a $t$-reduction can be applied.  
Suppose without loss of generality this is a $t$-reduction of the first kind listed above. Applying this $t$-reduction will result in a product of the form 
$p_0t^{\gamma_1}p_1t^{\gamma_2}\dots 
p_{i-1}
(p_i \phi)
p_{i+1}
\dots 
t^{\gamma_m}p_m$, where $\gamma_i=-1$, $p_i \in A$ and $\gamma_{i+1}=-1$. 
Since $t, t^{-1} \in M$ by assumption, it follows that $p_i \phi = t^{-1} p_i t \in M$ and hence $p_{i-1} p_i p_{i+1} \in M$. Hence the $G$-terms in the above product still all belong to $M$. 

The argument when a $t$-reduction of the second kind is applied is analogous. 
\end{proof}

In the following lemma we identify conditions under which  the process in Lemma \ref{lem:HNN:reduced} can be performed algorithmically. 

\begin{lem}\label{lem:HNN-eff-red}
Let $G^*=G *_{t,\phi:A\to B}$ such that  $G$, $A$ and $B$ are all finitely generated,  $G$ has recursively enumerable word problem,  and the membership problems for $A$ and $B$ within $G$ are both decidable.  Then there an algorithm which takes any word $w \in (\ol{X \cup \{t\}})^*$ as input and returns a word in reduced form 
\[
w= w_0t^{\eps_1}w_1t^{\eps_2}\dots t^{\eps_n}w_n,
\]
with $w_i \in \ol{X}^*$ for $1 \leq i \leq n$. 
\end{lem}
\begin{proof}
This may be proved by combining the comments from \cite[pages 184--185]{LSch} discussing conditions under which the process of conducting $t$-reductions is effective, together with Lemma~\ref{lem:display}. 
\end{proof}

\subsection{Proof of Theorem~\ref{thm:hnn41}}

A key ingredient in our proof is given by the following auxiliary result, which strengthens Lemma \ref{lem:HNN:reduced} under the stronger conditions of Theorem \ref{thm:hnn41}.

\begin{pro}\label{pro:hnn43}
Assuming all the notation and conditions from Theorem \ref{thm:hnn41}, let
$$
g = g_0t^{\eps_1}g_1t^{\eps_2}\dots t^{\eps_n}g_n
$$
be an element of $G^*=G *_{t,\phi:A\to B}$ written in reduced form. Then $g\in M$ if and only if $g_i\in M\cap G$ for all  $0\leq i\leq n$. 
\end{pro}

\begin{proof}
($\Rightarrow$) 
By Lemma~\ref{lem:HNN:reduced}, since $g\in M$ there exists a reduced form 
$$
g = m_0t^{\mu_1}m_1t^{\mu_2}\dots t^{\mu_k}m_k
$$
such that $m_i\in M$ for all $0\leq i\leq k$. By Lemma \ref{lem:Britton}, we must have $k=n$, $\mu_i=\eps_i$ and
$$
g_i = \alpha_i^{-1} m_i (t^{\eps_{i+1}}\alpha_{i+1}t^{-\eps_{i+1}})
$$
for all $0\leq i\leq n$ and some $1=\alpha_0,\alpha_1\,\dots,\alpha_n,\alpha_{n+1}=1\in A\cup B$ (such that $\alpha_i\in A$
whenever $\eps_i=-1$ and $\alpha_i\in B$ otherwise). It follows that $$g_i\in (A\cup B)(M\cap G)(A\cup B)\subseteq M\cap G.$$

($\Leftarrow$) is trivial, since $t,t^{-1}\in M$.
\end{proof}

\begin{proof}[Proof of Theorem \ref{thm:hnn41}]
To prove the theorem we must show that  there is an algorithm which takes any word $w$ from $(\ol{X \cup \{t \}})^*$ as input and decides whether or not the word represents an element of the submonoid $M$.  The hypotheses of Lemma~\ref{lem:HNN-eff-red} are clearly satisfied  Applying this lemma  we conclude that  there is an algorithm that given any such word $w$ returns a word 
\[
w= w_0t^{\eps_1}w_1t^{\eps_2}\dots t^{\eps_n}w_n,
\]
with $w_i \in \ol{X}^*$ for $1 \leq i \leq n$, that is in reduced form.  It follows from Proposition~\ref{pro:hnn43} that $w$ represents an element of $M$ if and only if $w_i \in M \cap G$ for all $0 \leq i \leq n$.  However, this is decidable by assumption (iii), and this completes the proof. 
\end{proof}

Combining Benois' Theorem and Britton's Lemma with Theorem~\ref{thm:hnn41} gives the following result. 

\begin{cor}\label{cor:hnn-free197} 
Let $G^*=G *_{t,\phi:A\to B}$ where $G = FG(X)$ is a free group of finite rank, and $A$ and $B$ are finitely generated subgroups of $G$. Let $T$ be a finitely generated submonoid of $G$ containing $A$, and let $M$ be the submonoid of $G^*$ generated by $T \cup \{t, t^{-1} \}$. 
Then the membership problem for $M$ in $G^*$ is decidable. 
\end{cor}

This can be applied to show that membership is decidable in certain finitely generated submonoids of surface groups. It is an open problem whether in general the submonoid membership problem is decidable for surface groups. 

\subsection{Proof of Theorem \ref{thm:hnn}}

The following lemma of combinatorial nature will turn out to be crucial in what follows.

\begin{lem}\label{lem:seq}
Let $G^*=G *_{t,\phi:A\to B}$ be an HNN extension of a finitely generated group $G$ with finitely generated associated subgroups $A,B$.  Let $G$ be finitely generated by $X$ with canonical homomorphism $\pi: \ol{X}^* \rightarrow G$.  Then there is an algorithm which takes any finite list of finite subsets  $W_0,W_1,\dots,W_d,W_1',\dots,W_d'$ of $G$  and any integer $m \geq 0$ and 
returns an integer $C_m\geq 1$ and  a finite set of FSA 
\[
\{  \mathcal{N}_{m,i}^{(j)} :\ 0 \leq i \leq m, 1 \leq j \leq C_m \}
\]
over $\ol{X}$  such that with $N_{m,i}^{(j)} = [L(\mathcal{N}_{m,i}^{(j)})] \pi$ 
and
\[
D_m =  \left\{ h_0 t h_1 t \dots t h_m :  (h_0, \dots, h_m) \in  \bigcup_{1 \leq j \leq C_m} \left(  N_{m,0}^{(j)} \times \dots \times N_{m,m}^{(j)} \right) \right\}, 
\]
we have 
\[
\Mon\gen{W_0\cup W_1t\cup W_2t^2\cup\dots\cup W_dt^d\cup tW_1'\cup\dots\cup t^dW_d'} = \bigcup_{m \geq 0} D_m.
\]
\end{lem}
\begin{proof}
We will first describe the algorithm, and then verify that the sets $D_m$ do indeed satisfy the property claimed in the statement of the lemma. 

The algorithm we describe will in fact output rational expressions for the sets $N_{m,i}^{(j)}$. By the comments in Section~\ref{sec:prelim}, this suffices to conclude that there is an algorithm computing the corresponding FSA 
$\mathcal{N}_{m,i}^{(j)}$. 

The algorithm is recursive.  When $m=0$ the algorithm returns $C_0=1$ and  the rational expression $W_0^*$ so that $N_{0,0}^{(1)}=\Mon\gen{W_0}$. 
Now consider a typical stage $m$ of the algorithm with $m>0$, assuming that the algorithm already constructed the integers $C_p$ and rational expressions for sets $N_{p,i}^{(j)}$, $1\leq i\leq p$, $1\leq j\leq C_p$, for all values $p<m$.

The algorithm then proceeds as follows. For any $1\leq\mu\leq \min(d,m)$ construct the following two collections
of sequences of subsets of $G$ (all of length $m+1$):
\begin{itemize}
    \item $\Mon\gen{W_0}W_\mu,\underbrace{\{1\},\dots,\{1\}}_{\mu-1},N_{m-\mu,0}^{(j)},\dots,N_{m-\mu,m-\mu}^{(j)}$,
    \item $\Mon\gen{W_0},\underbrace{\{1\},\dots,\{1\}}_{\mu-1},W_\mu'N_{m-\mu,0}^{(j)},\dots,N_{m-\mu,m-\mu}^{(j)}$,
\end{itemize}
where in both cases $1\leq j\leq C_{m-\mu}$. 
Set 
$$
C_m = \sum_{\mu=1}^{\min(d,m)} 2C_{m-\mu}.
$$
That is, $C_m$ is the total number of sequences constructed above. 
Then set $N_{m,i}^{(j)}$ to be the set appearing in position $i$ of the $j$th sequence in the above list,
$1\leq j\leq C_m$. 
The algorithm computes the above sequences, and the number $C_m$. It is clear from the definition of these sequences all the sets appearing in these sequences are rational subsets of $G$, and the algorithm can be instructed to output the appropriate corresponding rational expressions.

To complete the proof of the lemma we must now verify that  
\[
\Mon\gen{W_0\cup W_1t\cup W_2t^2\cup\dots\cup W_dt^d\cup tW_1'\cup\dots\cup t^dW_d'} = \bigcup_{m \geq 0} D_m
\]
holds. 

Set $K =  \Mon\gen{W_0\cup W_1t\cup W_2t^2\cup\dots\cup W_dt^d\cup tW_1'\cup\dots\cup t^dW_d'}$.  For all $m \geq 0$ let $K_m$ be the subset of $K$ of all elements that when written in reduced form have precisely $m$ occurrences of $t$.  This is well-defined by Britton's Lemma and $K$ is a disjoint union of the sets $K_m$ with $m \geq 0$. Note that it also follows from Britton's Lemma that $\bigcup_{m \geq 0} D_m$ is a union of pairwise disjoint sets. 

Hence to  finish the proof of the lemma it will suffice to prove that 
$K_m = D_m$ for all $m \geq 0$, which we shall prove by induction on $m$.  

The base case $m=0$ holds because $K_0=\Mon\gen{W_0}=D_0$ where the second equality follows by the construction of the algorithm.

For the induction step, consider $K_m$ and $D_m$ and suppose by induction that $K_{m’} = D_{m’}$ for all $0 \leq m' < m$.    

To see that $K_m \subseteq D_m$, let $g \in K_m$ be arbitrary. By the definition of $K$  
$$
g = u_1\dots u_k,
$$
where each term $u_r$, $1\leq r\leq k$, is either of the form $w_rt^{\delta_r}$ for some $w_r\in W_{\delta_r}$, or of the form 
$t^{\delta_r}w_r$ for some $w_r\in W_{\delta_r}'$, where $0\leq\delta_r\leq d$. 

Since there is no occurrence of $t^{-1}$ in $u_1 \dots u_k$ it follows that this is in reduced form thus,  as $g \in K_m$,   it follows from Britton's Lemma  that  $\delta_1+\dots+\delta_k=m$. 

Let $s$ be the smallest index such that in the above decomposition of $g$ we have  $\delta_s=\mu>0$.  
This implies that $u_1\dots u_{s-1}\in \Mon\gen{W_0}$ and $\mu\leq \min(d,m)$ because $\delta_1+\dots+\delta_k=m$.
As for $u_s$, we have
either $u_s\in W_\mu t^\mu$, or $u_s\in t^\mu W'_\mu$. Finally, $u_{s+1}\dots u_k$ is an element of the monoid $K$ with the property that any of its reduced forms has precisely $m-\mu$ occurrences of $t$. Therefore,
$u_{s+1}\dots u_k\in K_{m-\mu}$, which by induction hypothesis implies that $u_{s+1}\dots u_k\in D_{m-\mu}$.
We conclude that either
$$
g \in \Mon\gen{W_0}W_\mu t^\mu D_{m-\mu},
$$
or 
$$
g \in \Mon\gen{W_0}t^\mu W'_\mu D_{m-\mu}.
$$
In both cases, by the description of our algorithm, we have 
$$ 
g\in N_{m,0}^{(j)}tN_{m,1}^{(j)}t\dots tN_{m,m}^{(j)}
$$ 
for some $1\leq j\leq C_m$, yielding $g\in D_m$.

Conversely, to see that $D_m \subseteq K_m$, let $g\in D_m$. Then there exists an index $j$, $1\leq j\leq C_m$, such that $g\in N_{m,0}^{(j)}tN_{m,1}^{(j)}t\dots tN_{m,m}^{(j)}$. Now, the sequence of sets $N_{m,0}^{(j)},N_{m,1}^{(j)},\dots,N_{m,m}^{(j)}$ is obtained in one of the two ways as described in the definition of our algorithm. Therefore, there is an integer $\mu$, $1\leq\mu\leq\min(d,m)$, such that  either $g \in \Mon\gen{W_0}W_\mu t^\mu D_{m-\mu}$, or $g \in \Mon\gen{W_0}t^\mu W'_\mu D_{m-\mu}$. However, by the induction hypothesis, $D_{m-\mu}=K_{m-\mu}$, so we have either $g \in \Mon\gen{W_0}W_\mu t^\mu K_{m-\mu}\subseteq K_m$, or $g \in \Mon\gen{W_0}t^\mu W'_\mu K_{m-\mu}\subseteq K_m$. In summary, in either case we conclude that $g\in K_m$, which completes our proof.
\end{proof}

The following result which gives necessary and sufficient conditions for an element in reduced form to belong to $M$, will be essential for the proof of Theorem~\ref{thm:hnn}.  

\begin{pro}\label{pro:inM}
Let $G^*=G *_{t,\phi:A\to B}$ be an HNN extension of a finitely generated group $G$ with finitely generated associated subgroups $A,B$. Let 
$$
M=\Mon\gen{W_0\cup W_1t\cup W_2t^2\cup\dots\cup W_dt^d\cup tW_1'\cup\dots\cup t^dW_d'}
$$ 
for some finite $W_0,W_1,\dots,W_d,W_1',\dots,W_d'\subseteq G$. In addition, for $n\geq 0$, let $C_n\geq 1$ and $N_{n,i}^{(j)}$, $0\leq i\leq n$, $1\leq j\leq C_n$, be the 
integers and rational subsets of $G$ given by Lemma~\ref{lem:seq}. 

Let 
\[
g = g_0t^{\eps_1}g_1t^{\eps_2}\dots t^{\eps_n}g_n,
\]
be an element of $G^*$ in reduced form where $g_i \in G$ for $0 \leq i \leq n$ and $\eps_j \in \{1,-1\}$ for  $1 \leq j \leq n$.   For $i \in \{0, \ldots, n\}$ and $1\leq j\leq C_n$ define subsets $Q_i^{(j)} = Q_i^{(j)}(g_0,\dots,g_n)$ of $G$ in the following way: 
\begin{align*}
Q_0^{(j)} &= \{1\},\\
Q_{i+1}^{(j)} &= ((N_{n,i}^{(j)})^{-1}Q_i^{(j)}g_i\cap A)\phi, \; \mbox{for $0 \leq i \leq n-1$.}
\end{align*}

Then $g \in M$ if and only if $\eps_1=\dots=\eps_n=1$ and 
$$
g_n \in \bigcup_{1 \leq j \leq C_n} (Q_n^{(j)})^{-1}N_{n,n}^{(j)}.
$$
\end{pro}
\begin{proof}
($\Rightarrow$) Assume that $g\in M$. By Lemma~\ref{lem:seq} there exist $m\geq 0$ and $1\leq j\leq C_m$ such that
$$
g = h_0th_1t\dots h_{m-1}th_m
$$
holds for some $h_i\in N_{m,i}^{(j)}$, $0\leq i\leq m$. Since the right-hand side of the previous equality 
is in reduced form and $g= g_0t^{\eps_1}g_1t^{\eps_2}\dots t^{\eps_n}g_n$ is a reduced form, by Lemma \ref{lem:Britton}, it follows that $m=n$, and $\eps_1 = \dots = \eps_n = 1$, 
and there exist $1=\alpha_0,\alpha_1,\dots,\alpha_n,\alpha_{n+1}=1\in B$ such that 
$$
g_i = \alpha_i^{-1}h_i(t\alpha_{i+1}t^{-1})
$$
holds for all $0\leq i\leq n$. In particular, we have
$g_0 = h_0(t\alpha_1t^{-1})$, so
$$
\alpha_1\phi^{-1}=t\alpha_1t^{-1}=h_0^{-1}g_0\in (N_{n,0}^{(j)})^{-1}g_0\cap A=(N_{n,0}^{(j)})^{-1}Q_0^{(j)}g_0\cap A,
$$
implying $\alpha_1\in ((N_{n,0}^{(j)})^{-1}Q_0^{(j)}g_0\cap A)\phi=Q_1^{(j)}$.

Now we proceed by induction to prove that for all $0\leq i\leq n$, $\alpha_i\in Q_i^{(j)}$. 
Suppose that $\alpha_k \in Q_k^{(j)}$ for some $0 \leq k \leq n$ and consider $\alpha_{k+1}$. We have 
$$
\alpha_{k+1} = (t\alpha_{k+1}t^{-1})\phi = (h_k^{-1}\alpha_kg_k)\phi \in ((N_{n,k}^{(j)})^{-1}Q_k^{(j)}g_k\cap A)\phi =  Q_{k+1}^{(j)},
$$
completing the induction step. 
It follows that $g_n=\alpha_n^{-1}h_n\in (Q_n^{(j)})^{-1}N_{n,n}^{(j)}$, as required.

($\Leftarrow$) 
Let 
$g = g_0t^{\eps_1}g_1t^{\eps_2}\dots t^{\eps_n}g_n$ 
be an element of $G^*$ in reduced form such that $\eps_i = 1$ for all $1 \leq i \leq n$ and 
$g_n\in (Q_n^{(j)})^{-1}N_{n,n}^{(j)}$ for some $1\leq j\leq C_n$. 
Then we can write $g_n=\beta_n^{-1}k_n$ for some 
$\beta_n\in Q_n^{(j)}$ and $k_n\in N_{n,n}^{(j)}$. Therefore
$$
t\beta_nt^{-1}=\beta_n\phi^{-1}\in Q_n^{(j)}\phi^{-1}= (N_{n,n-1}^{(j)})^{-1}Q_{n-1}^{(j)}g_{n-1}\cap A,
$$
by definition of $Q_n^{(j)}$. 
So there exist $k_{n-1}\in N_{n,n-1}^{(j)}$ and $\beta_{n-1}\in Q_{n-1}^{(j)}$ such that 
$t\beta_n t^{-1}=k_{n-1}^{-1}\beta_{n-1}g_{n-1}$. Rearranging this yields
$$
g_{n-1} = \beta_{n-1}^{-1}k_{n-1}(t\beta_n t^{-1}).
$$
Continuing in this way we obtain for $i = n-1, n-2, \ldots, 0$ elements  
$k_i\in N_{n,i}^{(j)}$ and $\beta_i\in Q_i^{(j)}$ ($0\leq i<n$) such that 
$$
g_i = \beta_i^{-1}k_i(t\beta_{i+1} t^{-1}). 
$$
Note that $\beta_0=1$ because $Q_0^{(j)}=\{1\}$. Substituting into the reduced form of $g$, and cancelling gives adjacent inverse pairs, we obtain 
$$
g = g_0tg_1t\dots tg_n = k_0tk_1t\dots tk_n \in M
$$
by Lemma~\ref{lem:seq}.
\end{proof}

\begin{lem}\label{lem:Qrat2}
Under the assumptions of Theorem~\ref{thm:hnn},
in the statement of Proposition~\ref{pro:inM}  
every set $Q_i^{(j)}$ ($0 \leq i \leq n$, $1 \leq j \leq C_n$) 
is a rational subset of $B$. 
\end{lem}
\begin{proof}
Clearly $Q_0^{(j)}$ is a rational subset of $B$. 
We prove by induction on $i$ (with $j$ fixed) that $Q_i^{(j)}$ is rational for $0 \leq i \leq n$.  
For the induction step suppose that $Q_k^{(j)}$ is rational and consider $Q_{k+1}^{(j)}$.  
By definition
\[
Q_{k+1}^{(j)} = ((N_{n,k}^{(j)})^{-1}Q_k^{(j)}g_k\cap A)\phi. 
\]
By Lemma~\ref{lem:seq} the set $N_{n,k}^{(j)}$ is a rational subset of $G$. 
It then follows from 
condition (ii) 
in Theorem~\ref{thm:hnn} 
and the induction hypothesis 
that  
\[
R = (N_{n,k}^{(j)})^{-1}Q_k^{(j)}g_k\cap A
\]
is a rational subset of $G$. 
Since $R \subseteq A$ it then follows from 
Theorem~\ref{thm:effective:herbst}
that $R \in \Rat(A) $. Then since $\phi$ is an isomorphism it follows that $Q_{k+1}^{(j)} = R \phi \in \Rat(B)$. 
\end{proof}

\begin{proof}[Proof of Theorem \ref{thm:hnn}] 
Similarly to the proof of Theorem~\ref{thm:hnn41}, to prove the theorem we must show that  there is an algorithm which takes any word $w$ from $(\ol{X \cup \{t \}})^*$ as input and decides whether or not the word represents an element of the submonoid $M$.    By assumption (i) it follows that the membership problems for $A$ in $G$ and $B$ in $G$ are both decidable, since $G$ has decidable rational subset membership problem and $A$ and $B$ are both finitely generated. Also, condition (i) implies that $G$ has decidable subgroup membership problem, and hence decidable word problem.  Hence, the hypotheses of Lemma~\ref{lem:HNN-eff-red} are satisfied.  Applying this lemma  we conclude that  there is an algorithm that given any such word $w$ returns a word 
\[
w= w_0t^{\eps_1}w_1t^{\eps_2}\dots t^{\eps_n}w_n,
\]
with $w_i \in \ol{X}^*$ for $1 \leq i \leq n$, 
and $\eps_i \in \{1,-1\}$ for $1 \leq i \leq n$, 
that is in reduced form.  

At this point, the algorithm calls as a subroutine the algorithm from Lemma~\ref{lem:seq} which will return 
an integer $C_n \geq 1$ and 
a finite set of FSA 
\[
\{  \mathcal{N}_{n,i}^{(j)} :\ 0 \leq i \leq n, 1 \leq j \leq C_n \}
\]
over $\ol{X}$  such that with $N_{n,i}^{(j)} = [L(\mathcal{N}_{n,i}^{(j)})] \pi$ 
the conditions in the statement of Lemma~\ref{lem:seq} are satisfied.  

Set $g_i = w_i \pi$ for $0 \leq i \leq n$. 
For each $i \in \{0, \ldots, n \}$ and $j \in \{1, \dots, C_n \}$
let $Q_i^{(j)} = Q_i^{(j)}(g_0,\dots,g_n)$ be defined as in the statement of Proposition~\ref{pro:inM}.
Then by Lemma~\ref{lem:Qrat2} each of these sets $Q_i^{(j)}$ is a rational subset of $B$, and therefore also a rational subset of $G$.  
\begin{claim}
There exists an algorithm which for each 
$i \in \{0, \ldots, n\}$ and $j\in\{1,\dots,C_n\}$ computes 
\begin{itemize}
    \item a finite state automaton $\mathcal{G}_i^{(j)}$ over $\ol{X}$ with 
    $[L(\mathcal{G}_i^{(j)})]\pi = Q_i^{(j)}$, and
    \item a finite state automaton $\mathcal{B}_i^{(j)}$ over $\ol{Z}$ with 
    $[L(\mathcal{B}_i^{(j)})]\xi = Q_i^{(j)}$.
\end{itemize}
\end{claim}
\begin{proof}[Proof of claim.] 
For each $1 \leq j \leq C_n$, 
the algorithm iteratively constructs the pairs $(\mathcal{G}_i^{(j)}, \mathcal{B}_i^{(j)})$ in the following way. When $i=0$ we have $Q_i^{(j)} = \{ 1 \}$ and it is clear that an appropriate pair $(\mathcal{G}_0^{(j)}, \mathcal{B}_0^{(j)})$ can be computed e.g.\ by taking automata that accept only the empty word in each case. 
Now consider a typical stage $i$ with $i >0$. 

Then by definition 
\[
Q_{i}^{(j)} = ((N_{n,i-1}^{(j)})^{-1}Q_{i-1}^{(j)}g_{i-1}\cap A)\phi. 
\]
The algorithm constructs the automaton $\mathcal{M}_{n,i-1}^{(j)}$ over $\ol{X}$ such that $[L(\mathcal{M}_{n,i-1}^{(j)})]\pi=(N_{n,i-1}^{(j)})^{-1}$ from the automaton $\mathcal{N}_{n,i-1}^{(j)}$. Using $\mathcal{M}_{n,i-1}^{(j)}$ and $\mathcal{G}_{i-1}^{(j)}$ the algorithm then produces, in the obvious way, a FSA $\mathcal{D}_i^{(j)}$ over $\ol{X}$   such that  $[L(\mathcal{D}_i^{(j)})] \pi = (N_{n,i-1}^{(j)})^{-1}Q_{i-1}^{(j)}g_{i-1}$.
The algorithm given by assumption (ii), in the statement of the theorem,  is then applied to the automaton $\mathcal{D}_i^{(j)}$ which yields an automaton  $\mathcal{C}_i^{(j)}$ over $\ol{X}$  satisfying 
\[
[L(\mathcal{C}_i^{(j)})]\pi =  (N_{n,i-1}^{(j)})^{-1}Q_{i-1}^{(j)}g_{i-1} \cap A.
\]
The algorithm then calls as a subroutine the algorithm from Theorem~\ref{thm:effective:herbst} which returns a FSA $\mathcal{A}_i^{(j)}$ over $\ol{Y}$ satisfying  
\[
[L(\mathcal{A}_i^{(j)})]\theta =  (N_{n,i-1}^{(j)})^{-1}Q_{i-1}^{(j)}g_{i-1} \cap A.
\]
By replacing each letter $y_l\in \ol{Y}$, $1\leq l\leq k$, in the transitions of the automaton $\mathcal{A}_i^{(j)}$ by the corresponding letter $z_l\in \ol{Z}$, the algorithm computes the automaton $\mathcal{B}_i^{(j)}$ over $\ol{Z}$
such that 
\[
[L(\mathcal{B}_i^{(j)})]\xi =  ((N_{n,i-1}^{(j)})^{-1}Q_{i-1}^{(j)}g_{i-1} \cap A)\phi = Q_i^{(j)}.
\]
Finally, 
the algorithm calls 
Theorem~\ref{thm:effective:herbst} as a subroutine which returns the automaton $\mathcal{G}_i^{(j)}$ over $\ol{X}$ satisfying
\[
[L(\mathcal{G}_i^{(j)})]\pi =  ((N_{n,i-1}^{(j)})^{-1}Q_{i-1}^{(j)}g_{i-1} \cap A)\phi = Q_i^{(j)}.
\]
This completes the proof of the claim.
\end{proof}

To complete the proof, by Proposition~\ref{pro:inM}, we have $g \in M$ if and only if $\eps_1=\dots=\eps_n=1$ and $g_n \in (Q_n^{(j)})^{-1}N_{n,n}^{(j)}$.  Using the automata $\mathcal{G}_{n}^{(j)}$ and $\mathcal{N}_{n,n}^{(j)}$, the algorithm produces, in the obvious way,  a FSA $\mathcal{A}(w)$ over $\ol{X}$ such that $[L(\mathcal{A}(w))]\pi = (Q_n^{(j)})^{-1}N_{n,n}^{(j)}$. 

Therefore, in summary we have shown that there is an algorithm which given any word $w \in (\ol{X}\cup\{t\})^*$ computes a word $w_0t^{\eps_1}w_1t^{\eps_2}\dots t^{\eps_n}w_n$ in reduced form, equal to $w$ in $G$, and also computes 
a FSA $\mathcal{A}(w)$ over $\ol{X}$ such that $w$ represents an element of $M$ if and only if $\eps_1=\dots=\eps_n=1$ and $w_n \pi \in [L(\mathcal{A}(w))]\pi$. This is decidable by condition (i) of the theorem.  
\end{proof}

\begin{cor}\label{thm:hnn-free}
For a finite alphabet $X$ let 
$$
G = FG(X)*_{t,\phi:A\to B}
$$
be an HNN extension such that $A,B$ are finitely generated. 
Then for any finite subsets $W_0,W_1,\dots,W_d,W_1',\dots,W_d'$ of $FG(X)$ the membership problem for 
$$
M=\Mon\gen{W_0\cup W_1t\cup W_2t^2\cup\dots\cup W_dt^d\cup tW_1'\cup\dots\cup t^dW_d'}
$$ 
is decidable.
\end{cor}
\begin{proof}
It will suffice to prove that the hypotheses of Theorem~\ref{thm:hnn} are satisfied. 
Since $G$ is a free group, hypothesis 
(i) is satisfied by Benois' Theorem (Theorem \ref{thm:ben}) and Corollary \ref{cor:ben},
while condition (ii) holds as a consequence of \cite[Theorem 3.3 and Corollary 3.4]{BS}.
\end{proof}

\section{Applications of HNN extension results to the prefix membership problem}
\label{sec:appl2}

In this section we present several applications of the results from the previous section, mainly Theorem \ref{thm:hnn}
and its Corollary \ref{thm:hnn-free}.
These applications include, among others, 
some presentations of Adjan type, conjugacy pinched presentations and, in particular, Baumslag-Solitar groups.

\subsection{Exponent sum zero theorem}

For a word $w\in\ol{X}^*$ and $x^\eps\in\ol{X}$, where
$x \in X$ and 
$\eps\in\{1,-1\}$, we write 
$|w|_{x^\eps}$ for the number of occurrences of $x^\eps$ in $w$. For $t\in X$, the \emph{exponent sum} of $t$ in $w$ is 
the number $|w|_t-|w|_{t^{-1}}$. We say that $t$ has \emph{exponent sum zero} in $w$ if $|w|_t=|w|_{t^{-1}}\neq 0$.

We now describe a well-known method due to McCool and Schupp \cite{McCSch} for expressing certain one-relator groups as HNN extensions of one-relator groups with a shorter defining relator. 
See \cite[Page 198]{LSch}.

Let $w\in\ol{X}^*$ be a word in which the letter $t \in X$ has exponent sum zero. We define a word $\rho_t(w)$ over the infinite alphabet
$$\Xi=\{x_l:\ x\in X\setminus\{t\}, l\in\mathbb{Z}\}$$
obtained from $w$ by first replacing each occurrence of $x\in X\setminus\{t\}$ by $x_{-i}$ where $i$ is the exponent sum of $t$ 
in the prefix of $w$ preceding the considered occurrence of $x$,
and then deleting every occurrence of $t$. 
For each  $x\in X\setminus\{t\}$ let $\mu_x$ and $m_x$ be respectively the smallest and the greatest value of $j$ such that  $x_j$ actually appears in $\rho_t(w)$.
For example 
\[
\rho_t(bt^{-1} a t^2 b t^{-1} a) = b_0 a_1 b_{-1} a_0, 
\]
and $\mu_a = 0$ and $m_a = 1$. 

The following result is originally due to Moldavanski\u{\i} \cite{Mol}. Its proof can be extracted from the proof of 
Freiheitssatz given in \cite[Section IV.5]{LSch} which follows the approach in \cite{McCSch}.

\begin{pro}\label{pro:mol}
Let $w\in\ol{X}^*$ be a word in which $t\in X$ has exponent sum zero such that $\rho_t(w)$ is cyclically reduced. Then 
the group $G=\Gp\pre{X}{w=1}$ is an HNN extension of the group 
$$
H = \Gp\pre{\Xi_w}{\rho_t(w)=1}
$$
where $\Xi_w=\{x_l:\ x\in X\setminus\{t\}, \mu_x\leq l\leq m_x\}$. The associated subgroups $A$ and $B$ in this extension
are free groups freely generated by $\Xi_w\setminus\{x_{m_x}:\ x\in X\setminus\{t\}\}$ and $\Xi_w\setminus
\{x_{\mu_x}:\ x\in X\setminus\{t\}\}$, respectively, with the isomorphism $\phi:x_i\mapsto x_{i+1}$ for all $x\in X\setminus\{t\}$
and $\mu_x\leq i<m_x$.
\end{pro}

We say that $w$ is \emph{prefix $t$-positive} if $|u|_t-|u|_{t^{-1}}\geq 0$ for all prefixes $u$ of $w$. 
Analogously, $w$ is said to be \emph{prefix $t$-negative} if $|u|_t-|u|_{t^{-1}}\leq 0$ for all prefixes $u$ of $w$. 

\begin{thm}\label{thm:pos-neg}
Maintaining the notation above,  
let $G=\pre{X}{w=1}$ be a one-relator group presentation such that some $t\in X$ has exponent sum zero in $w$ and that $\rho_t(w)$
is cyclically reduced, 
and let 
$$
H = \Gp\pre{\Xi_w}{\rho_t(w)=1}. 
$$
Let $A$ be the subgroup of $H$ generated by 
$\Xi_w\setminus\{x_{m_x}:\ x\in X\setminus\{t\}\}$.
Suppose that 
\begin{itemize}
    \item[(i)] the rational subset membership problem is decidable in $H$, and 
    \item[(ii)] $A \leq H$ is effectively closed for rational intersections.   
\end{itemize}
(In particular conditions (i) and (ii) both hold in the case that $H$ is a free group.)
If $w$ is either prefix $t$-positive or prefix $t$-negative, then the group $G$ defined by the presentation 
$\Gp\pre{X}{w=1}$ has decidable prefix membership problem.
Consequently, if these conditions hold and 
the inverse monoid $\Inv\pre{X}{w=1}$
is $E$-unitary, then it has decidable word problem.
\end{thm}

\begin{proof}
First of all, the group $G=\Gp\pre{X}{w=1}$ is an HNN extension of $H$ by Proposition \ref{pro:mol}. So, to prove the theorem,
it suffices to show that the prefix monoid $P_w\subseteq G$ has a generating set of the form given in Theorem \ref{thm:hnn}.

We consider only the prefix $t$-positive case, the prefix $t$-negative case being analogous. Write 
$$
w \equiv 
u_0 
\tau_1
u_1 
\tau_2
u_2 \dots 
\tau_n
u_n,
$$
where $\tau_i \in \{t, t^{-1} \}^+$ for all $1 \leq i \leq n$ and $u_i \in \ol{X\setminus\{t\}}^*$,
where the words $u_1,\dots,u_{n-1}$ are all non-empty. 
Set $\xi_i$ to be the $t$-exponent sum of $\tau_i$ for all $1 \leq i \leq n$. 

The exponent sum zero condition translates to
$$
\xi_1 + \xi_2 +\dots+ \xi_n = 0,
$$
while it is easy to see that the prefix $t$-positive condition is implies that the inequalities
$$
\sigma_r = \xi_1 +\dots+ \xi_r \geq 0
$$
hold for all $1\leq r\leq n$. 
This implies that for any letter $x \neq t$ appearing in $w$ we have
$\mu_x,m_x\leq 0$ and $\mu_x\leq -\sigma_r\leq m_x$ for all $1\leq r\leq n$.
In particular, $\sigma_1=\xi_1>0$. 

Now we have
$$
\rho_t(w)\equiv u_0'u_1'\dots u_{n-1}'u_n',
$$
where both $u_0',u_n'$ are obtained from $u_0,u_n$, respectively, by equipping each of their letters by the subscript $0$, 
and for all $1\leq i\leq n-1$, the word $u_i'$
is obtained from $u_i$ by putting $-\sigma_i$ in the subscript of each letter in $u_i$. 

The prefixes $p$ of $w$ can be classified into the following three types:
\begin{itemize}
\item[(1)] $p$ is a prefix of $u_0$;
\item[(2)] $p\equiv u_0 \tau_1\dots \tau_i q$, where $q$ is a prefix of $u_i$ for some $1\leq i\leq n$;
\item[(3)] $p\equiv u_0 \tau_1 \dots u_{i-1} \theta_i$ 
for some 
$1\leq i\leq n$ and 
prefix $\theta_i$ of $\tau_i$.
\end{itemize}
Our aim is to rewrite the elements of $G$ represented by these prefixes with respect 
to the generating set $\Xi_w \cup \{t \}$ of the presentation of $G$ given in Proposition~\ref{pro:mol} as an HNN-extension $H*_{t,\phi:A\to B}$.  

Indeed, upon recalling 
from \cite[Page~198]{LSch}
that the generator $x_{-s}$ of $H$ represents the element $t^sxt^{-s}$ in 
$G = H*_{t,\phi:A\to B}$,   
we immediately see that the prefixes of $u_0$ translate into prefixes
of $u_0'$, a finite subset of $H$. 
This deals with the prefixes of type (1).
For prefixes of type (2), in $G$ we have the equality
$$p = u_0(t^{\sigma_1}u_1t^{-\sigma_1})\dots(t^{\sigma_{i-1}}u_{i-1}t^{-\sigma_{i-1}})(t^{\sigma_i}qt^{-\sigma_i})t^{\sigma_i},$$
which means that 
in the presentation for the HNN
extension 
$H*_{t,\phi:A\to B}$
we have 
$$p = u_0'\dots u_{i-1}'q't^{\sigma_i}$$ 
where 
$q'$ is the word obtained 
by putting the subscript $-\sigma_i$ on every letter of $q$, and thus 
$q'$ is a prefix of $u_i'$. 
By a very similar argument, prefixes of type (3) are expressed as 
$$u_0'\dots u_{i-1}'t^{\sigma_{i-1}} \theta_i.$$
Let 
$\alpha_i \in \mathbb{Z}$ be the $t$-exponent sum of $\theta_i$.  
By the prefix $t$-positive condition, $\sigma_{i-1}+\alpha_i\geq 0$.
This implies that $t^{\sigma_{i-1}} \theta_i \in \{t, t^{-1} \}^*$ has non-negative $t$-exponent sum 
and therefore is equal in $G$ to a non-negative power of $t$. 
Hence, upon defining $d=
\max_{1\leq i\leq n}\sigma_i$ we conclude that the monoid $P_w$ has a generating set of the form 
$$
W_0\cup W_1t\cup\dots\cup W_dt^d
$$
for some finite $W_0,W_1,\dots,W_d\subseteq H$. We now see that all the requirements of Theorem \ref{thm:hnn} are satisfied,
so we conclude that the membership problem of $P_w$ in $G$ is decidable. 
\end{proof}

\begin{rmk}
It is natural to compare the prefix $t$-positive condition in Theorem~\ref{thm:pos-neg} with of $w$-strictly positive presentations considered in \cite{IMM} where it was shown that 
groups defined by such 
presentations have decidable prefix membership problem; see \cite[Theorem~5.1]{IMM}. 
In \cite[Corollary~5.2]{IMM} it is shown that if $w$ is a cyclically reduced word 
such that $\Gp\pre{X}{w=1}$
is a $w$-strictly positive presentation 
then the group of units of 
$\Inv\pre{X}{w=1}$ is trivial.  
In contrast, the $t$-positive condition in our theorem certainly does not imply that the group of units is trivial, and in this way we see that the class of examples to which Theorem~\ref{thm:pos-neg} applies is distinct from those dealt with by \cite[Theorem~5.1]{IMM}. 
For example, the inverse monoid presentation 
\[
\Inv\pre{a,b,t}{ tat^{-1}btat^{-1}=1 }
\]
is $t$-positive and it may be shown the the group of units of this inverse monoid is the infinite cyclic group. Indeed it may be shown that $tat^{-1}$ and $b$ are the minimal invertible pieces of this relator, since it is easily seen that this inverse monoid is not a group.
These pieces satisfy the unique marker letter property, and hence by 
results proved in 
\cite{GR} the group of units of this inverse monoid 
isomorphic to the group 
defined by the presentation  
$\Gp\pre{x,y}{xyx=1}$. Hence the group of units of this monoid is the infinite cyclic group.  
\end{rmk}

\begin{exa}
Let
$$
M = \Inv\pre{a,b,c,t}{t^{-1}atcbt^{-2}at^2cbt^{-3}at^3c=1}.
$$
Then $t$ has exponent sum zero in the relator word, and, furthermore, this word is prefix $t$-negative. The corresponding group 
$$G=\Gp\pre{a,b,c,t}{(t^{-1}at)cb(t^{-2}at^2)cb(t^{-3}at^3)c=1}$$ is an HNN extension of the group 
$$
H=\Gp\pre{a_1,a_2,a_3,b_0,c_0}{a_1c_0b_0a_2c_0b_0a_3c_0=1}
$$
with $A=\Gp\gen{a_1,a_2}$, $B=\Gp\gen{a_2,a_3}$, and $a_i\phi=t^{-1}a_it=a_{i+1}$ for $i\in\{1,2\}$. 
The defining relator in the presentation of $H$ is a positive word and hence is cyclically reduced. 
Also, since the generator $a_1$ only appears once in that word, it follows from the 
Freiheitssatz that that $H$ is a free group of finite rank. 
Hence, by Corollary~\ref{thm:hnn-free} and Theorem~\ref{thm:pos-neg}
the group defined by  
$$G=\Gp\pre{a,b,c,t}{(t^{-1}at)cb(t^{-2}at^2)cb(t^{-3}at^3)c=1}$$ 
has decidable prefix membership problem. 
Since the monoid $M$ is $E$-unitary, as the defining relator word is cyclically reduced, it follows that $M$ has decidable word problem.  
\end{exa}

\begin{exa}
For a slightly more involved example, let $G=\Inv\pre{a,b,c,t}{w=1}$, where
$$
w\equiv tbcbt^8bbct^{-6}ct^{-3}at^3bt^{-3}at^3ct^{-2}ct^{-1}.
$$
Note that $w$ is not cyclically reduced; however, $t$ has exponent zero in $w$ and it is $t$-positive. Furthermore,
$$
\rho_t(w)\equiv b_{-1}c_{-1}b_{-1}b_{-9}^2c_{-9}c_{-3}a_0b_{-3}a_0c_{-3}c_{-1}
$$
is a cyclically reduced word, so G is an HNN extension of 
$$
H = \Gp\pre{a_0,b_{-9},\dots,b_{-1},c_{-9},\dots,c_{-1}}{\rho_t(w)=1}.
$$
Note that $\rho_t(w)$ is a cyclically reduced word. 
Since the generator $b_{-2}$ occurs only once in $\rho_t(w)$ it follows 
by the Freiheitssatz 
that $H$ is a free group of finite rank. 
Therefore, Theorem~\ref{thm:pos-neg} tells us that the membership problem for $P_w$ in $G$ is decidable. 
\end{exa}

As in the previous section, we now exhibit an example 
to which the methods of this section do not apply. 

\begin{exa}
Consider the presentation
$$
\pre{a,b,t}{bt^{-1}at^2bt^{-1}a=1}.
$$
Note that the 
relator word 
in this presentation 
is cyclically reduced and has exponent sum zero for the letter $t$.
However, it is neither prefix $t$-positive, 
nor prefix $t$-negative.
The group $G$ defined by this presentation is an HNN extension of $H=\Gp\pre{a_0,a_1,b_{-1},b_0}{b_0a_1b_{-1}a_0=1}$, which is
clearly a free group of rank 3. The associated subgroups $A=\Gp\gen{a_0,b_{-1}}$ and $B=\Gp\gen{a_1,b_0}$ are free groups of rank 2. 
However, upon identifying all the prefixes of $bt^{-1}at^2bt^{-1}a$ and expressing them in 
terms of the generators of the 
described HNN extension 
of $H$, we see that 
$$P_w=\Mon\gen{W_0 \cup W_1t\cup W_{-1}t^{-1}},$$ 
where
$$
W_0 = \{a_0,b_0,(b_{-1}a_0)^{-1}\}, \quad W_1=\{a_0,(b_{-1}a_0)^{-1}\}, \quad W_{-1}=\{b_0,(b_{-1}a_0)^{-1}\}.
$$
Now we cannot apply Corollary \ref{thm:hnn-free} because of the `mixed' nature of the generating set of $P_w$ which contains
both elements with $t$ and with $t^{-1}$. 
The underlying problem now is that when we form an arbitrary product of such elements
(that is, a product representing an element of $P_w$), we cannot guarantee anymore that such a product is already in reduced form, as we had in Lemma \ref{lem:seq} and Proposition \ref{pro:inM}. 
Also, keeping track of the rationality of subsets containing elements of $H$ occurring between consecutive  instances of $t$ and $t^{-1}$ in such products becomes increasingly troublesome as we are forced to make more and more  (potentially nested) $t$-reductions.
\end{exa}

\begin{exa}
In Example \ref{exa:surface} we have seen that the orientable surface group of genus $n\geq 2$, defined by its standard presentation 
$$
G_n=\Gp\pre{a_1,\dots,a_n,b_1,\dots,b_n}{[a_1,b_1][a_2,b_2]\dots[a_n,b_n]=1},
$$ 
has decidable prefix membership problem. Now, by using Theorem~\ref{thm:pos-neg} we can 
apply out results to give a new proof of 
\cite[Theorem 5.3(b)]{IMM} showing  
that the prefix membership problem is decidable for all cyclic conjugates of the relator word in the above presentation.

Indeed, upon denoting $u\equiv [a_2,b_2]\dots[a_n,b_n]$, we have four cases to consider:
\begin{itemize}
\item[(i)] $w\equiv a_1^{-1}b_1^{-1}a_1b_1u$;
\item[(ii)] $w\equiv b_1^{-1}a_1b_1ua_1^{-1}$;
\item[(iii)] $w\equiv a_1b_1ua_1^{-1}b_1^{-1}$;
\item[(iv)] $w\equiv b_1ua_1^{-1}b_1^{-1}a_1$.
\end{itemize}
Case (i) is already resolved in Example \ref{exa:surface}; 
to illustrate how to deal with the remaining ones, we consider Case (iii) the other cases being similar.
Take $a_1$ to be the stable letter. The word $w$ is cyclically reduced, $a_1$ is exponent sum zero in $w$, and
$w$ is $a_1$-positive. We conclude that $G_n$ is an HNN extension of 
$$
H = \Gp\pre{(b_1)_{-1},(b_1)_0,(a_2)_{-1},(b_2)_{-1},\dots}{(b_1)_{-1}v((b_1)_0)^{-1} = 1},
$$
where $v$ is obtained from $u$ by replacing each $a_i,b_i,\dots$ by $(a_i)_{-1},(b_i)_{-1},\dots$, respectively, for $2\leq i\leq n$. 
So, $H$ is a free group of rank $2n-1$ with associated cyclic subgroups generated by
$(b_1)_{-1}$ and $(b_1)_0=(b_1)_{-1}v$, respectively. By Theorem~\ref{thm:pos-neg} we obtain that $P_w$ has decidable membership in $G_n$.
\end{exa}

We finish this subsection by presenting yet another application of Theorem~\ref{thm:pos-neg} which concerns the prefix membership problem 
for one-relator groups defined by Adjan presentations \cite{Inam,MMSu} over a two-letter alphabet. Recall that a one-relator 
group, inverse monoid, or monoid presentation is an \emph{Adjan presentation} if it is of the form $\pre{X}{u=v}$, 
where $u,v\in X^*$ are positive words such that the first letters of $u,v$ are different, and also the last 
letters of $u,v$ are different. For our purposes, group presentations of Adjan type will be written as 
$\pre{X}{uv^{-1}=1}$; note that the given conditions ensure that the word $uv^{-1}$ is cyclically reduced.

\begin{thm}\label{thm:adjan}
Let $G=\Gp\pre{a,b}{uv^{-1}=1}$ be a group defined by an Adjan presentation such that $|u|_a=|v|_a$.
Assume that at least one of the following conditions hold:
\begin{itemize}
\item[(i)] one of the words $u$ or $v$ begins with $ba$;
\item[(ii)] one of the words $u$ or $v$ end with $ab$;
\item[(iii)] there exists an integer $k$, $1\leq k < |u|_a$, such that there is a single $b$ between the 
$k$th and the $(k+1)$th occurrence of $a$ in one of the words $u,v$, while in the other word the $k$th and 
the $(k+1)$th occurrence of $a$ are adjacent.
\end{itemize}
Then the prefix membership problem for $G$ is decidable, as is the word problem
for the inverse monoid $\Inv\pre{a,b}{uv^{-1}=1}$.
\end{thm}

\begin{proof}
For each of the assumptions (i)--(iii), there are four cases to consider depending upon the first and last 
letters of $u,v$. However, all these cases are very similar, so we consider only one of them in each instance.
Let $p=|u|_a=|v|_a$.

We begin by assuming that $u$ begins with $ba$ and ends with $a$. Then $v$ begins with $a$ and ends with $b$,
so we may write
\begin{align*}
u & = bab^{\alpha_2}\dots b^{\alpha_p}a,\\
v & = ab^{\beta_1}\dots ab^{\beta_p+1},
\end{align*}
for some integers $\alpha_i,\beta_i\geq 0$. The word
$$
uv^{-1} = bab^{\alpha_2}\dots b^{\alpha_p}ab^{-(\beta_p+1)}a^{-1}\dots b^{-\beta_1}a^{-1},
$$
has exponent sum zero for $a$ and is prefix $a$-positive.
By considering $a$ as the stable letter, it follows that $G$ is an HNN extension of the group
$$
H = \Gp\pre{b_0,b_{-1},\dots,b_{-p}}{b_0b_{-1}^{\alpha_2}\dots b_{-(p-1)}^{\alpha_p}b_{-p}^{-(\beta_p+1)}\dots b_{-1}^{-\beta_1}=1}.
$$
The defining relator is a cyclically reduced word, and $H$ is a free group of finite rank by the 
Freiheitssatz since the generator 
$b_0$ arises exactly once in the defining relator. 
Hence in this case the result follows by  
Theorem~\ref{thm:pos-neg}.

Similarly, if e.g.\ $v$ begins with $b$ and ends with $ab$ then $u$ both begins and ends with $a$, and so we may write
\begin{align*}
u & = ab^{\alpha_1}\dots b^{\alpha_{p-1}}a,\\
v & = b^{\beta_0+1}ab^{\beta_1}\dots ab,
\end{align*}
for some integers $\alpha_i,\beta_i\geq 0$. In this case, we conclude that $G$ is an HNN extension of the group
$$
\Gp\pre{b_0,b_{-1},\dots,b_{-p}}{b_{-1}^{\alpha_1}\dots b_{-(p-1)}^{\alpha_{p-1}}b_{-p}^{-1}\dots b_{-1}^{-\beta_1}b_0^{-(\beta_0+1)}=1}.
$$
The defining relator is a cyclically reduced word, and $H$ is a free group of finite rank by the  Freiheitssatz since the generator  $b_{-p}$ arises exactly once in the defining relator.  Hence in this case the result follows by   Theorem~\ref{thm:pos-neg}.

Finally, upon assuming (iii), let us further assume that $u$ begins and ends with $a$, while $v$ begins and ends with $b$.
Then, for example,
\begin{align*}
u & = ab^{\alpha_1}\dots b^{\alpha_{k-1}}abab^{\alpha_{k+1}}\dots b^{\alpha_{p-1}}a,\\
v & = b^{\beta_0+1}ab^{\beta_1}\dots b^{\beta_{k-1}}aab^{\beta_{k+1}}\dots ab^{\beta_p+1},
\end{align*}
or some integers $\alpha_i,\beta_i\geq 0$. This leads to the conclusion that $G$ is an HNN extension of the group
$$
\Gp\pre{b_0,b_{-1},\dots,b_{-p}}{b_{-1}^{\alpha_1}\dots  b_{-k} \dots b_{-p}^{-(\beta_p+1)}\dots b_{-(k+1)}^{-\beta_{k+1}}b_{-(k-1)}^{-\beta_{k-1}} \dots b_0^{-(\beta_0+1)}=1}.
$$
The defining relator is a cyclically reduced word, and $H$ is a free group of finite rank by the  Freiheitssatz since the generator  $b_{-1}$ arises exactly once in the defining relator.  Hence in this case the result follows by   Theorem~\ref{thm:pos-neg}.
\end{proof}

\begin{rmk}
There are examples to which Theorem~\ref{thm:adjan} applies, which are not handled in 
\cite[Corollary 2.6]{MMSu}. For example, it covers a part of Case 4 from that corollary for which the decidability of the prefix membership problem is not deduced there (one of the  simplest examples is $u\equiv aba$, $v\equiv baab$). 
This shows that our results are not consequences of the approach via distortion functions pursued in \cite{MMSu}.
\end{rmk}

\subsection{Conjugacy pinched presentations}

The ``HNN analogue'' of the class of cyclically pinched groups are the \emph{conjugacy pinched groups}: these are one-relator
groups defined by a presentation of the form
$$
\Gp\pre{X\cup\{t\}}{t^{-1}ut = v},
$$ 
where $u,v\in \ol{X}^*$ are nonempty reduced words. Again, for our purposes, conjugacy pinched group 
presentations will be written in the form 
$$\Gp\pre{X\cup\{t\}}{t^{-1}utv^{-1}=1}.$$ 

\begin{thm}\label{thm:pinch2}
The prefix membership problem is decidable for any group defined by a conjugacy pinched presentation 
\[\Gp\pre{X\cup\{t\}}{t^{-1}utv^{-1}=1}.\] 
Consequently, the word problem is decidable for all one-relator inverse monoids of the form    
\[\Inv\pre{X\cup\{t\}}{t^{-1}utv^{-1}=1}\] 
with $u$ and $v$ both reduced reduced words from $\ol{X}^*$.  
\end{thm}

\begin{proof}
By the Freiheitssatz,
any conjugacy pinched group is the HNN extension of the free group $FG(X)$ with associated
cyclic subgroups $A=\Gp\gen{u}$ and $B=\Gp\gen{v}$. Hence, to prove the theorem it suffices to compute the set of prefixes
of the word $w\equiv t^{-1}utv^{-1}$ (which generate the prefix monoid $P_w$) and see that it has the form required by 
Theorem~\ref{thm:hnn}. 
Indeed, we have 
$$
\pref(w) = t^{-1}\cdot\pref(u) \cup t^{-1}ut\cdot\pref(v^{-1}).       
$$
Note that in $G$ we have $t^{-1}ut\cdot\pref(v^{-1})=\pref(v)$, so $P_w$ is generated by $W_0\cup t^{-1}W_1'$ for
$W_0=\pref(v)$ and $W_1'=\pref(u)$, whence the required result follows (see Remark~\ref{rmk:inv}). 
\end{proof}

\begin{exa}
As an application of the previous theorem, 
groups defined by presentations of the form 
$$
B(m,n) = \Gp\pre{a,b}{b^{-1}a^mba^{-n}=1}
$$
have decidable prefix membership problems. 
These are so-called Baumslag-Solitar presentations.  
Hence, the inverse monoids 
$$\Inv\pre{a,b}{b^{-1}a^mba^{-n}=1}$$
have decidable word problems (cf.\ \cite[Theorem 4.2]{Inam} for a highly related result).
\end{exa}

\section{An undecidability result in the non-cyclically reduced case}
\label{sec:conclude}

In this article the main applications of our results have been to show that the prefix membership problem is decidable 
for certain groups defined by one-relator presentations. On the other hand, in the recent paper \cite{Gr-Inv} a word $w$ 
(over a 3-element alphabet $X$) is constructed such that the inverse monoid $M=\Inv\pre{X}{w=1}$ has undecidable word problem. 
Furthermore, it was proved in Theorem 3.8 of the same paper that $M$ is actually $E$-unitary. Combining this fact with 
\cite[Theorem 3.1]{IMM} (see Theorem~\ref{thm:pw-wp} for the statement) it follows that there does exists a one-relator group 
$G=\Gp\pre{X}{w=1}$ with undecidable prefix membership problem. Hence, the following open problem arises naturally.

\begin{prb}
Characterise the words $w\in\ol{X}^*$ with the property that the prefix membership problem for $\Gp\pre{X}{w=1}$ is decidable.
In particular, is the prefix membership problem decidable when $w$ is a cyclically reduced word?
\end{prb}

The latter question was stated in \cite[Question~13.10]{Bes-list}.
By modifying some ideas and results from \cite{Gr-Inv}, we shall now show that if one weakens the hypothesis of this problem 
to insisting only that $w$ is a reduced word, then this question has a negative answer. 

\begin{thm}\label{thm:undec}
There is a finite alphabet $X$ and a reduced word $w \in \ol{X}^*$ such that 
$\Gp\pre{X}{w=1}$  has undecidable prefix membership problem. 
\end{thm}
\begin{proof}
Let $H = \Gp\pre{a,b}{ab ab^{-1} a^{-1}b a^{-1}b^{-1} = 1}$. It follows from \cite[Theorem~2.4]{Gr-Inv} 
that there is a finite set of words $u_1, u_2, \ldots, u_k \in \ol{\{a,b\}}^*$ such that the membership problem for 
$$T =\Mon\gen{u_1, u_2, \ldots, u_k}$$ 
in $H$ is undecidable. Set $r \equiv ab ab^{-1} a^{-1}b a^{-1}b^{-1}$ and $s \equiv a^{-1} b^{-1} ab ab^{-1} a^{-1}b$, and let  
\[
\beta \equiv 
(a r a^{-1}) 
(b r b^{-1}) 
(a^{-1} s a) 
(b^{-1} s b) 
\]
and 
\[
\gamma \equiv 
(t u_1 t^{-1}) r (t u_1^{-1} t^{-1}) r 
(t u_2 t^{-1}) r (t u_2^{-1} t^{-1}) r 
\dots 
r
(t u_k t^{-1}) r (t u_k^{-1} t^{-1}),
\]
where $t$ is a new letter not in $\ol{\{a,b\}}$.  Now define 
\[
w \equiv 
\beta \gamma r \gamma^{-1} \beta^{-1}.
\]
It is easy to see that $w$ is a reduced word in $\ol{X}^*$ where $X = \{a,b,t\}$. We claim that $\Gp\pre{X}{w=1}$  has undecidable 
prefix membership problem. Let $P = \Mon\gen{\mathrm{pref}(w)} \leq G$.  From the definition of $w$ it follows that $r=1$ in the group $G$. 
Since $r=1$ and $s$ is a cyclic conjugate of $r$, it follows that $s=1$ in $G$. 
Using the fact that $r=1$ and $s=1$, by considering the prefixes of $\beta$ we see that all of $a, a^{-1}, b$ and $b^{-1}$ belong to $P$ 
(meaning that the elements these words represent all belong to $P$). Since $\beta = 1$ in $P$, considering prefixes of $\gamma$ and using 
the fact that $r=1$ in $G$ we see that $t$ belongs to $P$, and $t u_i t^{-1}$ belongs to $P$ for all $1 \leq i \leq k$. Since every other 
prefix of $w$ is clearly expressible as a product of these elements we conclude that $P$ is equal to the submonoid of $G$ generated by 
\[
\ol{\{a,b\}} \cup \{t\} \cup \{ t u_i t^{-1} : i \in \{1, \ldots, k \} \}.
\]
It may be shown (see \cite[Lemma~3.6]{Gr-Inv}) that for any word $v \in \ol{\{a,b\}}^*$ we have that $tvt^{-1}$ represents an element of $P$ 
if and only if in $H$ the word $v$ represents an element in the submonoid $T \leq H$. By assumption the submonoid membership problem for $T$ 
in $H$ is undecidable, and hence it follows that the membership problem for $P$ within $G$ is undecidable. Hence $\Gp\pre{X}{w=1}$  has 
undecidable prefix membership problem, where $w \in \ol{X}^*$ is a reduced word. 
\end{proof}

\begin{rmk}\label{rmk:final}
Note that in the proof of Theorem~\ref{thm:undec} the initial presentation
\[
\Gp\pre{a,b}{ab ab^{-1} a^{-1}b a^{-1}b^{-1} = 1}
\]
for the group $H$ does have decidable prefix membership problem, and this follows as a consequence of Theorem~\ref{thm:pos-neg}. 
To see this, note that the letter $a$ has exponent sum zero in the word $r\equiv ab ab^{-1} a^{-1}b a^{-1}b^{-1}$. Furthermore, 
$r$ is prefix $a$-positive. Now following the method described in Proposition~\ref{pro:mol}, working with respect to the exponent 
sum zero letter $a$, the group $H$ arises as an HNN extension of the group
$$
H_1 = \Gp\pre{b_{-2},b_{-1},b_0}{b_{-1}b_{-2}^{-1}b_{-1}b_0^{-1}=1},
$$
which is just the free group of rank 2 generated by $b_{-2}$ and $b_{-1}$. Since $H_1$ is a free group it follows that the hypotheses 
(i) and (ii) of Theorem~\ref{thm:pos-neg} are both satisfied. Hence, Theorem~\ref{thm:pos-neg} can be applied and it follows that the 
above presentation for $H$ has decidable prefix membership problem.  

We conclude that the question of decidability of the prefix membership problem depends on the presentation of the considered group; in this 
remark and in the previous theorem we have just seen two presentations of the same group $H$, one yielding undecidable prefix membership problem, 
whereas the same problem is decidable with respect to the other presentation.
\end{rmk}

%


\begin{ackn}
The authors are grateful to an anonymous referee for a number of helpful comments. In particular, one of their suggestions led to the example
included in Remark \ref{rmk:final}. 
\end{ackn}


\end{document}